\documentclass[11pt]{amsart}
\usepackage{graphicx} 
\usepackage{amsthm}
\usepackage{amssymb}
\usepackage{amsmath}
\usepackage{xcolor}
\usepackage[style=ext-numeric,
            sorting=nyt,
            sortcites=true,
            maxnames=50,
            giveninits=true,
            articlein=false]{biblatex}
\usepackage{physics}
\usepackage[mathscr]{euscript}
\bibliography{Bibliography}
\usepackage{mathrsfs}
\usepackage{appendix}
\DeclareMathOperator*{\argmin}{arg\,min}

\usepackage{nicematrix}
\allowdisplaybreaks

\usepackage[margin=1.24in
]{geometry} 

\usepackage[bookmarks,bookmarksopen=false,
                  pdfauthor={Autor},%
                  pdftitle={Titel},%
                  colorlinks=true,%
                  linkcolor=blue,%
                  citecolor=blue,%
                  urlcolor=blue,%
                ]{hyperref}

\newtheorem{theorem}{Theorem}[section]
\newtheorem{prop}[theorem]{Proposition}
\newtheorem{corollary}[theorem]{Corollary}
\newtheorem{lemma}[theorem]{Lemma} 

\newtheorem{definition}[theorem]{Definition} 
\newtheorem{exmp}[theorem]{Example}

\newtheorem{remark}[theorem]{Remark} 

\usepackage{verbatim}

\newcommand{\A}{\mathfrak{M}}
\newcommand{\M}{\mathfrak{M}}
\newcommand{\V}{\mathfrak{V}}
\newcommand{\C}{\mathbb{C}}
\newcommand{\R}{\mathbb{R}}

\newcommand{\N}{\mathbb{N}}
\newcommand{\T}{\mathbb{T}}
\renewcommand{\S}{\mathcal{S}}

\DeclareFontFamily{U}{mathx}{}
\DeclareFontShape{U}{mathx}{m}{n}{<-> mathx10}{}
\DeclareSymbolFont{mathx}{U}{mathx}{m}{n}
\DeclareMathAccent{\widehat}{0}{mathx}{"70}
\DeclareMathAccent{\widecheck}{0}{mathx}{"71}

\newcommand{\w}{\gamma}
\renewcommand{\r}{\rho}

\title[On Quantum Large Sieve Inequalities]{On Quantum Large Sieve inequalities and  operator recovery from incomplete information}
\author{Lu\'is Daniel Abreu}
\address{Lu\'is Daniel Abreu, Faculty of Mathematics \\
	University of Vienna \\
	Oskar-Morgenstern-Platz 1 \\
	1090 Vienna, Austria}
    \email{abreuluisdaniel@gmail.com}
\author{Michael Speckbacher}
\address{Michael Speckbacher,	Acoustics Research Institute\\ Austrian Academy of Sciences\\ Dominikanerbastei 16, 1010 Vienna,  Austria}
\email{michael.speckbacher@oeaw.ac.at}
\author{Erling Svela}
\address{Erling Svela, Department of Mathematical Sciences, Norwegian University of Science and Technology, 7491 Trondheim,
Norway}
\email{erling.a.t.svela@ntnu.no}
\date{}

\begin{document}

\begin{abstract}
    We obtain large sieve type inequalities for the Rayleigh quotient of the
restriction of phase space representations of higher rank operators, via an
operator analogue of the short-time Fourier transform (STFT). The resulting
bounds are referred to as \emph{quantum large sieve inequalities}. Building on the foundational work of Donoho and Stark, we demonstrate that these inequalities guarantee the recovery of an operator whose phase-space information is missing or unobservable over a 'measure-sparse' region $\Omega $, by
solving an $L^{1}$-minimization program.
This is an operator version of what is commonly known as
`Logan's phenomenon'. 
Moreover, our results can be viewed as a deterministic, continuous variable operator version
of `low-rank' 
matrix recovery, which itself can be regarded as a (finite rank) operator version of compressive
sensing. Our results depend on an abstract large sieve principle for
operators with integrable STFT and on a non-commutative analogue of the local
reproducing formula in rotationally invariant domains (first stated by Seip for the Fock space of entire
functions). As an application, we obtain concentration estimates for Cohen's class distributions and the Husimi function. We motivate the paper by comparing  with 
Nicola and Tilli's Faber-Krahn inequality for the STFT, illustrating that norm bounds on a
domain $\Omega $, obtained by large sieve methods, introduce a \emph{trade-off}
between \emph{sparsity} and \emph{concentration }properties of $\Omega $: If 
$\Omega $ is sparse, the large sieve bound may significantly improve known
operator norm bounds, while if $\Omega $ is concentrated, it produces worse
bounds.
\end{abstract}
\subjclass[2020]{46E40, 81S30, 47B32, 47B10, 46E22, 11N36}
\keywords{Large sieve inequalities, operator-valued short-time Fourier transform, concentration estimates, quantum state tomography}

\maketitle

\section{Introduction}

While signal analysis typically deals with vectors and functions, quantum
physics often requires matrices or density operators of \emph{higher rank} to represent  \emph{mixed states}. Since
any normalized function $f$ can be interpreted as a \emph{rank-one} density operator via the mapping $f\mapsto f\otimes f$, signal analysis of functions can be regarded 
 as a particular instance of signal analysis of quantum states. 
 Phase space representations are extremely useful tools for analyzing functions and operators. Recent developments have led to new phase space representations for operators whose frequency properties change with time ~\cite{OpSTFT}.

The purpose of this paper is to obtain
large sieve type inequalities for the Rayleigh quotient of the restriction of  phase space representations of higher
rank  operators. The resulting bounds, which will be referred to as \emph{quantum large sieve inequalities}, provide guarantees for an operator-level analogue of Logan's phenomenon \cite{Logan,DonohoLogan,DonohoLogan}: the possibility of
recovering an operator whose phase space values
are missing or impossible to observe in a `measure-sparse' region $\Omega $, by solving an $L^{1}$-minimization program in the complementary region $\Omega ^{c}$.

The large sieve principle has its origins in number theory \cite{montgomery}%
, where it provides asymptotic averages of arithmetic functions on integers
constrained by congruences modulo sets of primes. It has been used in a
number of different mathematical fields. In particular, it found important
applications in signal analytic settings, first for  band-limited
functions \cite{DonohoLogan} and, more recently, for the short-time Fourier \cite{LargeSieve} and wavelet transforms \cite{ls-wavelets}, and band-limited spherical harmonics expansions \cite{sphere}. The large sieve
inequalities considered in \cite{DonohoLogan,LargeSieve,ls-wavelets,sphere}
consist of norm bounds of localization operators on a domain $\Omega $,
which introduce a \emph{trade-off} between \emph{sparsity} and \emph{%
concentration }properties of $\Omega $. If $\Omega $ is sparse, the large
sieve bound may significantly improve  known operator norm bounds, while if $\Omega $ is
concentrated, it produces worse bounds. Before presenting our results in the operator setting, we illustrate the key principles of our approach in Section~\ref{subsec:intro1} by revisiting Nicola and Tilli's `Faber-Krahn inequality' for the STFT \cite{faber-krahn}. While this inequality is optimal for arbitrary sets, we demonstrate that it can be significantly sharpened for sets whose Lebesgue measure is `well-spread' across the time-frequency plane.

To put the ideas of the previous paragraph on a firm mathematical foundation, we first present several definitions and illustrative examples. 
First, we define the concept of $R$\emph{-measure} (or \emph{maximum Nyquist density}) of a Euclidean set $\Omega
\subset \mathbb{R}^{2d}$ \cite{LargeSieve}:
\begin{equation*}
\nu (\Omega ,R)=\sup_{z\in \mathbb{R}^{2d}}|\Omega \cap D_R(z)|\text{.}
\end{equation*}
The $R$\emph{-measure} is close to the Lebesgue measure of the disk for a `well-concentrated' set $\Omega $, but it can be much smaller for a `spread-out' $\Omega $. Examples to support this claim can be easily constructed; see, e.g., the set \eqref{Rsparse}. Clearly, for a fixed $R>0$, and  any given $\epsilon,\mu >0$, one can always construct a
 set $\Omega$  with $|\Omega|=\mu $ and $\nu (\Omega,R)<\epsilon $. We will loosely refer to `$R$\emph{-sparse}' sets, as those for which  $\nu
(\Omega ,R)/\pi R^2$ is small, but we will not specify any threshold for this
quantity, since it does not enter our results directly.

Let $f,g\in L^{2}(\mathbb{R}^{d})$ and $g\neq 0$. For $z=(x,\omega )\in \mathbb{R}^{2d}$, the time-frequency shift of $g$ is $\pi (z)g(t)=g(t-x)e^{2\pi
i \omega \cdot t }$. The \emph{short-time Fourier transform (STFT)} of
$f$  with window  $g$ is defined as 
\begin{equation}
V_{g}f(z)=\langle f,\pi (z)g\rangle  \text{.
}  \label{STFT}
\end{equation}
In the terminology of signal analysis and physics, the operator $V_{g}$ transforms a function $f(t)$, defined over time (in signal analysis) or position (in physics) with $t\in \mathbb{R}^{d}$,  into a function $V_{g}f(z)$, where $z=(x,\omega )\in 
\mathbb{R}^{2d}$  represents the joint \emph{time-frequency} domain in signal analysis or the \emph{phase space} in physics.
In \cite{LargeSieve}, large sieve
inequalities were obtained for the STFT by estimating the Rayleigh
quotient of the Daubechies-type localization operator \cite{Daubechies,faber-krahn}, 
\begin{equation}
\Phi _{\Omega ,g}^{(p)}(f) =\frac{\int_{\Omega }|V_{g}f(z)|^{p}\,dz}{\int_{%
\mathbb{R}^{2d}}|V_{g}f(z)|^{p}\,dz },  \label{Rayleigh}
\end{equation}%
in terms of $\nu (\Omega ,R)$. 
In addition to the case $p=1$, estimates for the concentration problem with $p=2$ are also of particular interest, as they provide eigenvalue estimates for the Daubechies localization operator \(A_{\Omega}^g\) \cite{Daubechies,de2002uniform, CorderoGrochenig}, defined weakly as \begin{align*}
    A_{\Omega}^gf=\int_\Omega V_gf(z)\pi(z)g \, dz.
\end{align*} By the Courant-Fischer theorem, the value of the supremum of \eqref{Rayleigh} is the largest eigenvalue of \(A_\Omega^g\), a well-known connection which has been exploited with great success in the past; see, for instance, \cite{AccSpec,faber-krahn}.

In order to study an operator analogue of the Rayleigh quotient in~\eqref{Rayleigh} we need a way to represent an operator on phase space. We will use the operator STFT first studied in~\cite{OpSTFT}. 
Given a `window operator' $\w\in \S^{2}$ (where $\S^2$ denotes the space of Hilbert-Schmidt
operators), the operator STFT of $\r \in \S^{2}$ is defined as%
\begin{equation*}
\mathfrak{V}_{\w} \r (z):=\w^{\ast }\pi (z)^{\ast }\r
,\qquad z\in \mathbb{R}^{2d}\text{.}
\end{equation*}%
This transform preserves several of the convenient properties of the
classical STFT \eqref{STFT}.
We will provide some information in Section~\ref{subsec:intro2} and a more detailed exposition in Section~\ref{sec:prel}. The main  goal of this paper is to obtain large sieve
inequalities for the following natural operator analogue of the Rayleigh
quotient \eqref{Rayleigh}, 
\begin{equation}
 {\Phi }_{\Omega ,\w}^{(p)}(\r )=\frac{\int_{\Omega
}\left\Vert \mathfrak{V}_{\w}\r(z)\right\Vert _{\S^{2}}^{p}dz}{%
\int_{\mathbb{R}^{2d}}\left\Vert \mathfrak{V}_{\w}\r
(z)\right\Vert _{\S^{2}}^pdz}\text{,}  \label{RayleighOp}
\end{equation}%
in terms of $\nu (\Omega ,R)$. 
Let $\A^1\subset \S^2$ be the space of operators whose operator STFT has finite $L^1(\R^{2d},\S^2)$ norm, first introduced in ~\cite{OpSTFT}. As for the classical STFT case, we will see that once the condition $%
\sup_{\rho\in\A^1} {\Phi }_{\Omega ,\w}^{(1)}(\r )<\frac{1}{2}$ is met  - 
 something that our inequalities assure if $\nu(\Omega,R)/\pi R^2$ is small enough - then $L^1$-minimization is guaranteed to exactly recover  $\mathfrak{V}_{%
\w}\r (z)$ for all $z\in \Omega $, even when the values of $\mathfrak{V}_{\w}\r (z)$ over the region $\Omega $ are missing and only those in the complement  
 $\Omega ^{c}$ are observed. Precisely, we have: 

\begin{equation*}
\rho =\argmin_{\sigma \in \mathfrak{M}^{1}}\left\Vert \V_\w\sigma \right\Vert
_{L^{1}(\mathbb{R}^{2d},\S^{2})},\qquad \text{ subject to }\mathfrak{V}%
_{\gamma }\sigma |_{\Omega ^{c}}=\mathfrak{V}_{\gamma }\rho |_{\Omega ^{c}}.
\end{equation*}
 More general recovery scenarios, involving the presence of noise, can be
considered. They will be
gathered in Section~\ref{sec:recovery}.
\medskip

The paper is structured as follows: This introduction is extended by two
subsections. In Section~\ref{subsec:intro1}, we revisit the large sieve for the (rank-one) STFT, and compare it to 
 the   sharp Faber-Krahn 
inequality \cite{faber-krahn} for different levels of $R$-sparsity of the concentration domain. This is mostly intended to
provide intuition for Section~\ref{subsec:intro2}, where we present our main results
regarding the estimation of \eqref{RayleighOp}. In most cases we present a particular form of our
results in $d=1$, to facilitate a first reading.  Section~\ref{sec:prel} contains  the required
background. In Section~\ref{sec:S^2=L^2} we study the best concentration problem in the case \(p=2\) and show that the solutions on \(\S^2\) coincide with those on \(L^2\). The proof of the operator large sieve is given in Section~\ref{sec:large-sieve}. Section~\ref{sec:estimating-theta} provides the proof of a novel kind of local reproducing formula, which holds for polyradial operators of higher rank. This formula is then combined with explicit special function estimates, in order to obtain fully explicit large sieve inequalities. Then, Section~\ref{sec:applications} considers two applications of the large sieve principle, first to the Husimi function, and then for distributions in the Cohen class. Finally, in
Section~\ref{sec:recovery},  we show how `Logan's phenomenon' and other $L^{1}$
-minimization reconstruction scenarios, can be adapted to the operator setting   \cite{DonohoLogan}.

\subsection{Large sieve inequalities for the STFT}\label{subsec:intro1}  
To simplify the presentation and to enable direct comparison with \cite{faber-krahn} we will fix   the 1-dimensional Gaussian window $g(t)=h_{0}(t)=2^{1/4}e^{-\pi t^{2}},\ t\in\R,$ in \eqref{STFT} throughout this section.  With this choice, we can identify the image of $L^{2}(\mathbb{R})$ under the STFT with the Fock
space of entire functions. Exploiting this structure, Nicola and Tilli  
\cite{faber-krahn} established   a \emph{sharp} estimate for the Rayleigh quotient \eqref{Rayleigh}: 
\begin{equation}
\Phi _{\Omega ,h_{0}}^{(p)}(f)\leq 1-e^{-p|\Omega |/2},\qquad \Omega\subset \R^2 \text{.}  \label{Faber-Krahn}
\end{equation}
This is a special case of the results in \cite{faber-krahn}, which confirmed a
conjecture of two of us in \cite[Conjecture~1]{LargeSieve}. Estimates for the Rayleigh quotient are of particular interest for signal recovery methods when $p=1$.
Assume, for example, that the \emph{certificate condition}
\begin{equation}
\sup_{f\in M^{1}(\R)}\Phi _{\Omega ,h_{0}}^{(1)}(f)<1/2  \label{1/2}
\end{equation}
is satisfied  (where $ M^{1}(\mathbb{R})$ denotes the modulation space \cite{feichtinger1983modulation}
of functions with $
L^{1}$-integrable STFT, also known as the Feichtinger algebra \cite{Fei}). Then one may perfectly recover $V_{h_0}f$  on the whole phase space from only observing 
 $V_{h_0}f(z)$ \emph{outside} $ \Omega 
$. In particular, we have
\begin{equation}
f=\argmin_{_{h\in M^{1}(\R)}}\Vert V_{h_0}h\Vert _{1}\text{,
\ \ \ \ \ subject to }V_{h_0}h|_{\Omega ^{c}}=V
_{h_0}f|_{\Omega ^{c}}\text{,}  \label{l1}
\end{equation}
 In the setting
of bandlimited functions \cite{Logan},\cite[Theorem~9]{DonohoStark}, this
remarkable property is known as `Logan's phenomenon', an embryonic idea of
the powerful signal analytic theory which became known as  compressive sensing \cite{CompressedBook,CandesRomb,SparsityTF}, extended to discrete density operators in \cite{Gross,GrossIEEE} and to time-dependent operators in \cite{schreiber2025tomography}. Logan's phenomenon has recently been  revisited in \cite{LargeSieve, ls-wavelets, husain2024concentration} and combined with restriction theory in \cite{iosevich2025uncertainty}.

The condition \eqref{1/2} is difficult to assure in general. Even the sharp result \eqref{Faber-Krahn} requires the restrictive condition $|\Omega |<2\log 2$. But
large sieve methods can produce better bounds than \eqref{Faber-Krahn} for $R$%
-sparse sets. The following observation, contained in \cite{LargeSieve}, is
the rank-one case of \eqref{eq:ab-ls} in the next section.
Let $\Omega \subset \mathbb{R}^{2}$, $g\in M^{1}(\R)$ and set%
\begin{equation*}
I_{R}(V_{h_0}f)(z):=\int_{D_R(z)}V_{h_0}f(w)\langle \pi(w)h_0,\pi(z)h_0\rangle dw,
\end{equation*}%
where $D_R(z)$ denotes the disk of radius $R$ and center $z\in\R^2$.
Then the
following inequality holds 
\begin{equation}
\Phi_{\Omega,h_0}^{(1)}(f)\leq \sup_{f\in M^{1}(\R)}\frac{\left\Vert V%
_{h_0}f\right\Vert _{1}}{\left\Vert I_{R}(V_{h_0}f)\right\Vert _{1}}%
\cdot \sup_{z\in \mathbb{R}^{2}}\int_{\Omega \cap D_R(z)}|V%
_{h_0}h_0(z-w)|dw\text{.}  \label{L_S}
\end{equation}%
The \emph{local reproducing
formula} (see \cite{Seip0} for the Gaussian and \cite{LargeSieve} for the extension to
general Hermite windows)
\begin{equation}\label{eq:loc-rep-gauss}
V_{h_{0}}f(z)=\frac{1}{1-e^{-\pi R^{2}}}\int_{D_R(z)}V%
_{h_{0}}f(w)\langle \pi(w)h_0,\pi(z)h_0\rangle dw,\qquad R>0,  
\end{equation}%
gives $\big(1-e^{-\pi R^{2}}\big)^{-1}$ for the first factor in \eqref{L_S}.\ Since $%
|V_{h_{0}}h_{0}(z-w)|=e^{-\pi |z-w|^{2}/2}$, a change of variables
in \eqref{Faber-Krahn} leads to a sharp estimate for the second factor in \eqref{L_S}:
\begin{equation}
\sup_{z\in \mathbb{R}^{2}}\int_{\Omega \cap D_R(z)}|V%
_{h_{0}}h_{0}(z-w)|dw\leq 2(1-e^{-\nu (\Omega ,R)/2})\text{.}  \label{FKOp}
\end{equation}%
This results in the following large sieve inequality, which seems to have been
hitherto unnoticed: 
\begin{equation}
\Phi _{\Omega ,h_{0}}^{(1)}(f)\leq \frac{2(1-e^{-\nu (\Omega ,R)/2})}{1-e^{-\pi
R^{2}}}\text{.}  \label{R-FK}
\end{equation}
For $R$-concentrated sets (i.e., $|\Omega |\approx \nu (\Omega
,R)$), \eqref{R-FK} is worse than \eqref{Faber-Krahn}, but for $R$-sparse
sets, \eqref{R-FK} can provide a significant  improvement upon \eqref{Faber-Krahn}. For example, if $
\Omega =D(0,R)$, then \eqref{R-FK} gives $\Phi _{\Omega ,h_{0}}^{(1)}(f)\leq 2$,
while, from \eqref{Faber-Krahn} we get $\Phi _{\Omega ,h_{0}}^{(1)}(f)\leq
1-e^{-\pi R^{2}/2}$. In contrast, for 
\begin{equation}
\Omega = \bigcup\limits_{k=0}^{N-1}D\left( 2kR,\frac{R}{2}
\right) \text{,}  \label{Rsparse}
\end{equation}
we have $\nu (\Omega ,R)=\pi R^{2}/4$, while $|\Omega |=N\pi R^{2}/4$ (which
includes the case $N=\infty $). In this case, \eqref{Faber-Krahn} gives $\Phi _{\Omega
,h_{0}}^{(1)}(f)\leq 1-e^{-N\pi R^{2}/8}$, requiring $NR^{2}<\frac{8\log 2}{
\pi }$\ to ensure the certificate condition \eqref{1/2}. On the other hand, \eqref{R-FK}
leads to
\begin{equation*}
\Phi _{\Omega,h_{0}}^{(1)}(f)\leq \frac{2(1-e^{-\pi R^{2}/8})}{1-e^{-\pi
R^{2}}}\text{,}
\end{equation*}
which is less than $1/2$ for $R<1/9$ and arbitrary $N$ (even when $N=\infty $ and $
|\Omega |=\infty $). 
It is actually easy to construct domains for which the Faber-Krahn bound is arbitrarily close to one while the large sieve bound is arbitrarily close to zero. 

\subsection{A large sieve inequality for the STFT of density operators}\label{subsec:intro2}

We will now describe
the main results of the paper, as well as some of the challenges faced if one aims to extend the observation of the previous section to density
operators.
As a first step, this requires extending the STFT from rank-one operators (functions) to higher rank operators. Our approach relies on an idea of Skrettingland~\cite{EqNorms}, which was considered in full generality in~\cite{OpSTFT}, leading to the introduction of the  \emph{operator STFT}. In this extension
the phase space is $L^2(\R^{2d}, \S^2)$ with the
Bochner inner product
$$\langle F,G\rangle_{L^2(\R^{2d},\S^2)} =\int_{\R^{2d}}\langle F(z),G(z)\rangle_{\S^2}dz.$$
  To gain some intuition on the definition of the operator STFT, let us recall the original idea of Skrettingland. If we let $ \w = g \otimes e$ and $\r= f \otimes  e$, $e, f, g \in L^2(\R^d)$ with $\|e\|_2=1$, then the
classical STFT can be found in the rank-one operator
$
\w^\ast \pi(z)^\ast \r = V_gf (z)\,  (e\,  \otimes\, e).$
In ~\cite{OpSTFT}, motivated by this, Dörfler, Luef, McNulty and Skrettingland  defined the operator STFT as
$$\V_\w \r(z) := \w^\ast \pi(z)^\ast \r,\qquad \w,\r\in\S^2,\ z \in\R^{2d}.
$$ The operator STFT can be seen as a non-commutative extension (for the algebra of operators instead of functions) of the classical  STFT and shares many properties with the classical STFT, which, as shown above, is recovered in the rank-one case. The range of the operator STFT space $  \V_\w (\S^2) \subset L^2( \R^{2d}, \S^2)$
is a reproducing kernel Hilbert space with  kernel
$K_\w(z, w) =\|\w\|_{\S^2}^{-2}\, \w^\ast\pi(z)^\ast\pi(w)\w$, and 
if  $\w$ is normalized in $S^2$, then it follows that $\V_\w : \S^2 \to  \V_\w (\S^2)$ is an isometry, see~\cite{OpSTFT} for details. Following~\cite{OpSTFT}, we denote
by $\mathfrak{M}^p$ the space of operators $\r$ with  $\V_\w \r\in L^p(\R^{2d}, \S^2)$.

The operator STFT encompasses a wide range of well-known time-frequency distributions. Beyond the classical STFT, recovered when both 
$
\w$ and 
 $\r$ are rank-one operators, familiar objects also arise when only one of them is assumed to be rank-one. For instance, if 
\(\r=f\otimes e\), then the pointwise Hilbert-Schmidt norm satisfies
$$\|\V_\w(f\otimes e)\|_{\S^2}^2=Q_{\w\w^*}f(z),$$ which is a distribution in Cohen’s class \cite[Section~7]{Bible2}. Similarly, if   \(\w=h_0\otimes h_0\) and \(\r\) is a general density operator, then $$ \big\|\V_{h_0\otimes h_0}\sqrt{\r}\big\|_{\S^2}^2=\langle \r\pi(z)h_0,\pi(z)h_0\rangle,$$ which corresponds to the Husimi function of the quantum state \(\r\) \cite{Husimi1}. Consequently, our main results can be interpreted in terms of familiar phase space representations. As we will see in Section~\ref{sec:applications}, the special structure of these two representations simplifies the analysis of the involved constants in the large sieve principle. 

 Similarly to the classical STFT, in order to establish a large sieve estimate we study the STFT concentration problem over \(\Omega\):\begin{align}\label{eq:OpSTFTconc}
  {\Phi }_{\Omega ,\w}^{(p)}(\r ) \leq  \sup_{\r\in \M^p}\frac{\int_{\Omega}\|\V_\w \r(z)\|_{\S^2}^pdz}{\int_{\R^{2d}}\|\V_\w\r(z)\|_{\S^2}^pdz}.
\end{align}
Before considering the case $p=1$, let us note that just like for the classical STFT, there is a connection between the concentration problem and the spectrum of a special type of operator, namely the (operator-valued) \textit{mixed-state localization operator \(A_\Omega^{\w\w^\ast}\)}~\cite{Bible2,OpSTFT}, defined as \begin{align*}
    A_\Omega^{\w\w^\ast}\r=\int_{\Omega}\pi(z)\w\V_\w\r(z)\,dz.
\end{align*} The rank-one case \(\r=f\otimes e\) recovers the function-valued mixed-state localization operator \begin{align*}
     A_\Omega^{\w\w^\ast}(f\otimes e)=\left(\chi_\Omega\star \w\w^*\right)\left(f\right)\otimes e=\left(A_{\Omega}^{\w\w^\ast}  f\right)\otimes e,
\end{align*} which is known to be related to the best concentration problem for Cohen's class \cite[Proposition~8.2]{Bible2}. Thus, we can think of the concentration problem \eqref{eq:OpSTFTconc} as a generalization of the concentration problem \begin{align*}
    \sup_{f\in L^2(\R^d)}\frac{\int_{\Omega}\|\V_\w f(z)\|_{\S^2}^2dz}{\int_{\R^{2d}}\|\V_\w f(z)\|_{\S^2}^2dz}=\sup_{f\in L^2(\R^d)}\frac{\int_{\Omega}Q_{\w\w^*}f(z)dz}{\int_{\R^{2d}}Q_{\w\w^*}f(z)dz}.
\end{align*} 
While the problem in \(L^2(\R^d)\) might seem more restrictive than the \(\M^p\) one, we will show that for \(p=2\), the optimal values of the two Rayleigh quotients actually coincide. This will be the topic of Section~\ref{sec:S^2=L^2}, where it will be shown that
\begin{equation*}
    \sup_{\r\in\S^2} \frac{\int_\Omega \|\V_\w\r(z)\|_{\S^2}^2dz}{\|\r\|_{\S^2}^2}=\sup_{f\in L^2(\R^d)} \frac{\int_\Omega \|\V_\w f(z)\|_{2}^2dz}{\|f\|_{2}^2}.
\end{equation*}
In Section~\ref{sec:large-sieve}, we  prove an abstract large sieve type  result of the form \eqref{eq:ab-ls} (stated in full generality in Proposition~\ref{prop:abstract-LS}). If we define the integral operator
   \begin{equation*}
I_\mathcal{K}(\V_\w\r)(z):= \int_{\R^{2d}}\mathcal{K}(z, ,w)\V_\w\r(w)dw 
    \end{equation*}
    for some adequately chosen kernel $\mathcal{K}(z,w)\in\S^2$, then for $1\le p<\infty$ 
    \begin{align}\label{eq:ab-ls}
      \Phi_{\Omega,\w}^{(p)}(\r)
      \leq \theta_\mathcal{K}\cdot\sup_{w\in \R^{2d}}\int_\Omega \left\|\mathcal{K}(z,w)\right\|_{\mathrm{op}}dz,
    \end{align} 
    with 
    $$
\theta_\mathcal{K}:=\sup_{\r\in\A^1}\left(\frac{\|\V_\w\r\|_{L^1(\R^{2d},\S^2)}}{\|I_\mathcal{K}(\V_\w \r)\|_{L^1(\R^{2d},\S^2)}}\right).
    $$
In Section~\ref{sec:estimating-theta}, we then provide explicit bounds for $\theta_\mathcal{K}$ for different choices of $\w$ and $\mathcal{K}$.
If, for example, we choose $d=1$ and $\w$ a finite rank polyradial operator, 
\begin{equation}\label{eq:finite-polyradial} \w =
\sum_{n=0}^N
\lambda_n(h_n \otimes h_n),
\end{equation}
(with $h_n$ denoting the $n$-th Hermite function)
then we obtain    the following fully explicit operator large sieve estimate for the Rayleigh quotient of the operator-valued mixed-state localization operator. The operator STFT  with the above windows shares some structural features with the   polyanalytic Bargmann transform  for vector-valued functions introduced in \cite{abreu2010sampling}, but its non commutativity poses additional challenges in the analysis. The non-analytic nature of the resulting spaces, together with the absence of an exact analogue of Seip's formula \eqref{eq:loc-rep-gauss}, considerably obstructs the adaptation of the methods described in Section  \ref{subsec:intro1}, which strongly depend on the properties of entire functions used in \cite{faber-krahn} and on the local reproducing formula. Our efforts to circumvent these obstructions led to the following.

\begin{theorem}\label{thm:1} Let $1\le p<\infty,$ $\w \in \mathfrak{M}^1$ be defined as \eqref{eq:finite-polyradial}, such that $\|\w\|_{\S^2} = 1 $, and let $\Omega\subset \R^{2}$ be
measurable.  If $\pi R^2=\alpha N$ for $\alpha\ge 5$, then  for every  $\r \in \mathfrak{M}^p$ it holds
$$
\Phi_{\Omega,\w}^{(p)}(\r)
\leq
\frac{\nu (\Omega ,R)}{1- \alpha^{2N}e^{N(2-\alpha)-\log(2)}}.
$$
\end{theorem}
Choosing $R$ according to the rank $N$ is 
reminiscent of the original large sieve principle for the Paley-Wiener space \cite{DonohoLogan} where the parameter $R$  is chosen according to the bandwidth. It is worth noting that the estimate improves with the rank $N$ of the density operator, suggesting a favorable outlook for the applicability of  $L^{1}$-minimization methods in the recovery of density operators from incomplete information.

The proof of Theorem~\ref{thm:1} is divided in three parts among two different sections, since the first two parts are required for other results. First, as already mentioned, we prove an abstract large sieve in  Proposition~\ref{prop:abstract-LS}. Then we prove Proposition~\ref{prop:quant-loc-rep}, an operator analogue of the local reproducing formula involving a multiplier. Finally, we estimate the multiplier in Example~\ref{ex:thm1}, among  other illustrative cases. Since the  operator analogue
of the local reproducing formula is considerably different from the rank-one case \eqref{eq:loc-rep-gauss}, we state it here for polyradial operators in a   simplified form (see Proposition~\ref{prop:quant-loc-rep} for the full details):
If   $\w$ is normalized in \(\S^2\) and, as in \eqref{eq:finite-polyradial}, we have
$$
\int_{D_R(z)}
K_\w (z, w)\V_\w\r (w)dw = \V_{\widetilde \w} \r (z),\qquad z\in\R^2,
$$
where $D_R(z)$ is the disk of radius $R$ centered at $z$, the operator $\widetilde \w
$ is defined by
$$\widetilde \w=
\sum_{
n=0}^\infty
\lambda_nA_n(R)(h_n\otimes h_n),
$$
and $A_n(R)$ is given explicitly in terms of sums of  integrals of generalized Laguerre polynomials which do not seem to admit relevant simplification. 
 The estimate required in Theorem~\ref{thm:1} follows from a lower bound on $A_n(R)$.

\begin{remark}
Note that, while the non-commutativity of the setting is not visible in the formulation of the  concentration problem, it can manifest itself in subtle ways. The local
reproducing formula stated above, for example, is a
non-commutative analogue of \eqref{eq:loc-rep-gauss} in the sense that $\widetilde{\gamma}$ is in general not a scalar multiple of $\gamma$, unless it is a rank-one operator. While in rank-one settings, the local reproducing formula is equivalent to double orthogonality in concentric domains \cite{Seip0,LargeSieve,ls-wavelets}, this is clearly not the case for higher rank operators (when we identify $h_n$ with $h_n\otimes e$, operator double orthogonality holds for Hermite functions, which as a result, are eigenfunctions of operator  localization operators in a classical sense   \cite{svela}, without any multiplier involved). In the discrete variables
setting of matrix recovery, the `low rank' of the matrix stands as a
non-commutative analogue (for the algebra of matrices) of `sparsity' of
vectors \cite{Gross,GrossIEEE}. See also~\cite[Section~5]{DKLMcN} for a related notion of sparsity for operators. It may be
interesting to observe that, in our continuous setting, sparsity for phase
space representations of operators, is measured in exactly the same way as
 sparsity for phase space representations of functions.  
\end{remark}

In the cases discussed so far, we were able to obtain good estimates of the
constant $\theta _{\mathcal{K}}$, but obtaining precise estimates for the
resulting second term in (\ref{eq:ab-ls}) seems to be out of reach in most
instances of higher rank density operators.\ But for  operators with the
following spectral decomposition (the square root of the so-called \emph{thermal
states,} see Section~\ref{sec:reproducing-kernel} for details and for the proof of Theorem~\ref{thm:2} below)%
\begin{equation}
\gamma =\frac{1}{\sqrt{1+a}}\sum_{n\in \mathbb{N}_{0}}\left( \frac{a}{a+1}%
\right) ^{n/2}(h_{n}\otimes h_{n})\text{,}  \label{squarethermal}
\end{equation}%
we can set $\theta _{\mathcal{K}}=1$ and obtain a good operator norm
estimate for ${\Phi }_{\Omega ,\gamma }^{(p)}$, by bounding the
second factor of  (\ref{eq:ab-ls}) in terms of the Hilbert-Schmidt norm of the reproducing kernel $%
{K}_{\gamma }(z,w)$ and using the Weyl correspondence. This leads to
the following result.

\begin{theorem}
\label{thm:2}Let $a>0$ and $\gamma $ as in (\ref{squarethermal}) and $1\le p<\infty$.\ For every 
$\rho \in \mathfrak{M}^{p}$,  
\begin{equation*}
\Phi_{\Omega,\w}^{(p)}(\r)\leq \frac{1}{\sqrt{1+2a}}\cdot \sup_{w\in \R^{2}}\int_\Omega  e^{-\frac{\pi}{2(1+2a)} |z-w|^2} \leq 2 \sqrt{1+2a}\Big(1-e^{-\frac{|\Omega |}{ 2(1+2a)}}\Big)\text{.}
\end{equation*}
\end{theorem}

\section{Preliminaries}\label{sec:prel}

\subsection{Notation}

 Without any subscripts $\langle\,\cdot \, ,\,\cdot\,\rangle$ will denote the inner product on $L^2(\R^d)$ with its corresponding norm $\|\cdot\|_2$. The other $L^p$-norms are denoted by \(\|\cdot\|_p\), $1\leq p\leq \infty$.
We use the convention   $f\otimes g=\langle \,\cdot\, ,g\rangle f$ for a   rank-one operator and denote the space of Schatten-$p$ operators acting on $L^2(\R^d)$ by  $\S^p$, $1\leq p\leq \infty$, where \(\S^\infty=\mathcal{B}(L^2(\R^d))\). We use  $\mathscr{S}(\R^d)$ for  the functions in Schwartz class, and define
 the space of Schwartz operators    \(\mathfrak{S}\)   as the space of integral operators on $L^2(\R^d)$  with integral kernel $k\in\mathscr{S}(\R^{2d})$. We denote its dual space, the tempered operators, by \(\mathfrak{S}'\). If an operator \(\r\) is positive  semi-definite, we write \(\r\succeq 0\).


\subsection{Basics of time-frequency analysis and Hermite functions}\label{sec:TFA}

We give a very brief overview of the most important concepts in time-frequency analysis, for further details see, e.g., \cite{Grochenig}.
Given \(f,g \in L^2(\R^d)\) and \(g\neq 0\), the \emph{short-time Fourier transform (STFT)} of \(f\) using the window \(g\) is defined by \begin{align*}
V_gf(z)=\int_{\R^d}f(t)\overline{g(t-x)}e^{-2\pi i \omega\cdot t} dt,\qquad z=(x,\omega)\in \R^{2d}.
\end{align*} 
If we define the time-frequency shift of \(g\) as \[\pi(z)g(t)=\pi(x,\omega)g(t)=g(t-x)e^{2\pi i  \omega\cdot t },\] we may write $V_gf(z)=\langle f,\pi(z)g\rangle.$ A fundamental property of the STFT is \emph{Moyal's identity}: 
\begin{equation}\label{eq:moyal}\int_{\mathbb{R}^{2d}}V_{g_1}f_1(z)\overline{V_{g_2}f_2(z)}dz=\langle f_1,f_2\rangle\overline{\langle g_1,g_2\rangle},\qquad f_1,f_2,g_1,g_2\in L^2(\R^d).
\end{equation}
 As a direct consequence, one may deduce that  $f$ can be reconstructed from $V_g f$ via
\begin{align*}
    f=\frac{1}{\| g\|^2_2}\int_{\R^{2d}} V_gf(z)\pi(z)g dz,
\end{align*} 
where  the integral has to be interpreted weakly. Another implication of \eqref{eq:moyal} is that $V_gf:L^2(\R^d)\to L^2(\R^{2d})$ is an isometry    if $\|g\|_2=1$. Moreover, the adjoint of $V_g$ is given by \[V_g^* F=\int_{\R^{2d}}F(z)\pi(z) gdz,\] where  again   the integral is to be interpreted weakly. If \(g\in \mathscr{S}(\R^d)\), then the definition of the STFT extends to a tempered distribution \(f\in \mathscr{S}'(\R^d)\) via the pairing \(V_gf(z)=\langle f,\pi(z)g \rangle_{\mathscr{S}',\mathscr{S}}\).

Let $h_0(t)=2^{d/4}e^{-\pi |t|^2}$ denote the standard Gaussian. We define the \emph{modulation space} $M^p(\R^d)$, $1\leq p\leq \infty$, as
$$
M^p(\R^d):=\left\{f\in\mathscr{S}'(\R^d):\ \|f\|_{M^p}:=\|V_{h_0} f\|_{p}<\infty\right\}.
$$ 
The \emph{(cross-)Wigner distribution} of \(f,g\in L^2(\R^d)\) is defined by \begin{align*}
    W(f,g)(x,\omega)=\int_{\R^d}f\Big(x+\frac{t}{2}\Big)\overline{g\Big(x-\frac{t}{2}\Big)}e^{-2\pi i \omega\cdot t}dt.
\end{align*}
We write \(W(f)=W(f,f)\).  Given a tempered distribution \(F\) on \(\R^{2d}\) one can define an operator \(L_F\), called the \emph{Weyl transform} of \(F\), via the pairing \begin{align*}
    \langle L_Ff,g\rangle=\langle F,W(g,f)\rangle,\qquad f,g\in\mathscr{S}(\R^d).
\end{align*}
 Given an operator \(T\), if there is a tempered distribution \(F\) on \(\R^{2d}\) such that the relation \begin{align*}
    \langle Tf, g \rangle=\langle F,W(g,f)\rangle,\qquad f,g\in\mathscr{S}(\R^d),
\end{align*}
holds, we call \(F\) the \emph{Weyl symbol} of \(T\).

The Hermite functions, defined by \begin{align*}
    h_n(t)=\frac{2^{1/4}}{\sqrt{n!}}\left(\frac{-1}{2\sqrt{\pi}}\right)^ne^{\pi t^2}\frac{d^n}{dt^n}\left(e^{-2\pi t^2}\right),\qquad t\in\R,
\end{align*} 
play an important role in time-frequency analysis and many related areas of mathematics. They are, for example, the eigenfunctions of the harmonic oscillator and form an orthonormal basis for $L^2(\R)$.
In higher dimensions we   consider the tensor product Hermite functions, defined by \[h_{k}(t_1,t_2,\dots,t_d)=\prod_{j=1}^d h_{k_j}(t_j),\qquad k\in\N_0^d.\]
Next, we define the \emph{generalized Laguerre polynomials} by  \begin{align}\label{eq:def-laguerre}
    L_k^\alpha(t)=\sum_{j=0}^k(-1)^j\binom{k+\alpha}{k-j}\frac{t^j}{j!},\qquad k\in\N_0,\  \alpha,t\geq 0.
\end{align}
Although the definition of the 
 generalized Laguerre polynomials  via \eqref{eq:def-laguerre} is only given for $\alpha\geq 0$, it may be extended for negative $\alpha$ using the reflection identity \begin{align}\label{Reflection}
    \frac{(-t)^n}{n!}L_{j}^{n-j}(t)=\frac{(-t)^j}{j!}L_{n}^{j-n}(t),
\end{align} 
and we will use this property throughout this paper to simplify notation. Like with the Hermite functions, we will   consider tensor product Laguerre polynomials, defined by \begin{align*}
    L_k^\alpha(t_1,t_2,\dots,t_d)=\prod_{j=1}^d L_{k_j}^{\alpha_j}(t_j),\qquad k\in\N_0^d,\,\alpha\in \R_{\geq 0}^d.
\end{align*}
The Laguerre polynomials provide us with a formula for the STFT of a Hermite function with respect to another Hermite function.
\begin{lemma}[Laguerre connection]
 Let $n,k\in\N_0$.   Then
    \begin{align}   V_{h_k}h_n(z)=V_{h_k}h_n(x,\omega)=\sigma_{n,k}(r)e^{i(k-n)\theta}e^{-\pi i xw},\qquad (x,\omega)\in\R^{2}
\end{align} where  \(\sigma_{n,k}(r)=\sqrt{\frac{k!}{n!}\pi^{n-k}}r^{n-k}L_k^{n-k}(\pi r^2)e^{-\pi r^2/2}\), and \(x=r\cos(\theta),\; \omega=r\sin(\theta)\).
\end{lemma}
Due to the tensor product structure of the Hermite and Laguerre functions the above formula  naturally extends to higher dimensions.

\subsection{The operator STFT}\label{sec:op-stft}

In this paper, we are mostly concerned with concentration inequalities for the following analogue of the STFT for Hilbert-Schmidt operators.
    Given \(\w,\r\in\S^2\), we define \(\mathfrak{V}_\w\r\) as the \(\S^2\)-valued function \begin{align}
      \mathfrak{V}_\w\r(z):=  \w^*\pi(z)^*\r,\qquad z\in \R^{2d}.
    \end{align} 
The operator STFT was first studied in~\cite{OpSTFTACHA}. Like the usual STFT, the \emph{operator STFT} represents a signal (an operator acting on functions on \(\R^d\)) by a function on the phase space \(\R^{2d}\). However, unlike the usual STFT, the operator STFT is operator-valued. Nevertheless, the operator STFT inherits many useful properties from the STFT. We collect the ones we need below. For further details we refer to \cite{OpSTFT}.

\begin{prop}\label{prop:op-stft}
    Let \(\sigma,\r,\w,\eta\in\S^2\), and \( F\in L^2(\R^{2d},\S^2)\). Then the following properties hold:
    \begin{enumerate}
        \item Moyal's identity: $$\int_{\R^{2d}}  \langle\V_\w\r(z),\V_\eta\sigma(z)\rangle_{\S^2}dz= \langle \eta,\w\rangle_{\S^2}\langle \r,\sigma\rangle_{\S^2} .$$
           \item The operator  \(\V_\w:\S^2\to L^2(\R^{2d},\S^2)\) is a multiple of an isometry,
          and its adjoint $\V_\w^\ast:L^2(\R^{2d},\S^2)\to\S^2$ is given by  $$\V_\w^\ast F=\int_{\R^{2d}}\pi(z)\w F(z)dz,$$ 
          where the integral is to be understood weakly in $\S^2$.
        \item  Inversion: $$\V_\w^*\V_\w\r=\int_{\R^{2d}}\pi(z)\w \V_\w\r(z)dz=\|\w\|_{\S^2}^2\, \r.$$
        \item The range  \(\V_\w(\S^2)\) is an \(\S^2\)-valued reproducing kernel Hilbert space with reproducing kernel $$
        K_\w(z,w)=\|\w\|_{\S^2}^{-2}\,\w^*\pi(z)^*\pi(w)\w.$$
    \end{enumerate}
\end{prop}

Using the singular value decomposition of a normalized  \(\w=\sum_{n=0}^\infty \lambda_n (f_n\otimes g_n)\), where  \(\{f_n\}_{n\in\N}\) and  \(\{g_n\}_{n\in\N}\) are two orthonormal sets, we may write the reproducing kernel \(\w^*\pi(z)^*\pi(w)\w\)  as 
\begin{align}\label{eq:kernel-written-out}
    \w^*\pi(z)^*\pi(w)\w=\sum_{m,n=0}^\infty \lambda_n\lambda_m\langle \pi(z) f_n,\pi(w) f_m\rangle (g_m\otimes g_n),
\end{align}
and obtain therefore a quadratic representation of operators in phase space, which may be an interesting alternative to the Wigner representation of operators \cite{cordero2024wigner}. If one  considers operators of the form \(\r=f\otimes g\), where $g\in L^2(\R^d)$ is normalized and fixed,  the operator STFT evaluates to \(\w^*\pi(z)^*(f\otimes g)=(\w^*\pi(z)^*f)\otimes g.\)  
It is therefore natural to define the STFT of a function \(f\) with respect to an operator \(\w\)  as the first term in the operator STFT of its corresponding rank-one operator \(f\otimes g\). Indeed, we define \(\mathfrak{V}_\w f\) as the \(L^2(\R^d)\)-valued function \begin{align}
      \mathfrak{V}_\w f(z):=  \w^*\pi(z)^*f,\qquad z\in\R^{2d}.
    \end{align} 
See \cite{OpSTFT} and \cite{EqNorms} for a detailed analysis of this transform and note that the properties of Proposition~\ref{prop:op-stft} also hold with minor adjustments.

As we are mainly interested in concentration estimates,  we note that it follows from the basic 
identity $\|\V_\w f(z)\|_2=\|\V_\w (f\otimes g)(z)\|_{\S^2}$ that 
\begin{equation}\label{eq:norm-S^2=norm-L^2}
\int_{\R^{2d}}\|\V_\w f(z)\|_2^pd\mu(z)=\int_{\R^{2d}}\|\V_\w (f\otimes g)(z)\|_{\S^2}^pd\mu(z),
\end{equation}
where $1\leq p\leq \infty$, and $\mu$ is an arbitrary measure on $\R^{2d}.$
Therefore, all results that we derive   for the operator-valued STFT  immediately carry over to the function-valued STFT.

\subsection{Quantum harmonic analysis and mixed-state localization operators}
On several occasions in this paper we will encounter convolutions between operators, and between operators and functions. These are the central objects of  the quantum harmonic analysis framework, first studied by Werner \cite{Werner}. The subject has attracted substantial interest in recent years; see, for example,~\cite{halvdansson2023quantum,Fulsche,QHABall,AffineQHA,LieQHA,luef2018convolutions}. We recall the basic notions here. For details and connections to time-frequency analysis we refer to \cite{Bible2}. 

For $z\in\R^{2d}$ and $T\in \mathcal{B}(L^2(\R^d))$ we define the \emph{translation of an operator} $T$ by
$$
\alpha_z(T)=\pi(z)T\pi(z)^\ast,
$$
and the  \emph{involution of an operator} $T$ via
$$
\widecheck{T}=PTP,
$$
where $P$ is the parity operator $Pf(t)=f(-t).$   
For future reference, we note that   $$\alpha_z(\alpha_w(T))=\alpha_{z+w}(T),\qquad   (\widecheck{T})\, \widecheck{\ } =T\qquad \text{and}\qquad  \alpha_{-z}(\widecheck{T})=P\alpha_z(T)P.$$ The latter property follows directly from the simple observations that $P\pi(z)=\pi(-z)P$ and $\pi(z)^\ast P=P\pi(-z)^\ast$.
The \emph{convolution of two operators} $S,T\in \mathcal{S}^1$
is the function $S\star T\in L^1(\R^{2d})$ defined as  
$$
S\star T(z)=\text{tr}(S\alpha_z(\widecheck{T})),\qquad z\in\R^{2d},
$$
while the convolution of a function $F\in L^1(\R^{2d})$ and an operator $S\in \S^1$ is the operator $F\star S\in \S^1$ given by 
$$
F\star S=\int_{\R^{2d}}F(z)\alpha_z(S)dz.
$$
We will later need that the convolution between two operators can be interpreted as the ordinary convolution between their respective Weyl symbols~\cite[Proposition~2.2]{AccCC}, i.e., for $F,G\in L^1(\R^{2d}),$ 
\begin{align}\label{eq:convolution-op-vs-symbol}
       L_F\star L_G(z)=F\ast G(z),\qquad z\in\R^{2d}.
    \end{align}
Although initially only defined for \(L^1\)-functions and trace class operators, the convolutions extend to the other Lebesgue spaces and Schatten classes via the \textit{Werner-Young inequalities}: (\cite[Proposition~3.6]{Bible2} or \cite{Werner}) \begin{align}\label{eq:werner-young}
    \|S\star T\|_{r}\leq \|S\|_{\S^p}\|T\|_{\S^q},
  \qquad\text{and}\qquad   \|F\star S\|_{\S^r}\leq \|F\|_{p}\|S\|_{\S^q},
\end{align}
where \(1\le p,q,r\le \infty\) satisfy \(\frac{1}{p}+\frac{1}{q}=1+\frac{1}{r}\). By duality, the two convolutions may also be extended to tempered distributions and operators in \(\mathfrak{S}'\), see \cite{SchwartzOp}, or \cite[Proposition~3.16]{Bible2}. 

The operator-operator convolution is connected to the operator STFT in the following way:
\begin{align}
  \|\V_\w\r\|_{\S^2}^2 & =\mathrm{tr}\left(\w^*\pi(z)^*\r\r^*\pi(z)\w\right)= \mathrm{tr}\left(\pi(z)\w\w^*\pi(z)^*\r\r^*\right)\notag
  \\
  &= \mathrm{tr}\left(\alpha_z(\w\w^*)\r\r^*\right)=\widecheck{\w\w^*}\star \r\r^*(z).\label{comp:cohen}
\end{align}
For future reference, we state this computation as a lemma. 
\begin{lemma}\label{Identity}
   Let \(\w, \r\in\S^2\).   Then we have\begin{align*}
        \big\|\V_\w\r(z)\big\|_{\S^2}=\sqrt{\widecheck{\w\w^*}\star \r\r^*(z)},\qquad z\in\R^{2d}.
    \end{align*}
\end{lemma}


For a positive  operator $\sigma$ with $\|\sigma\|_{\S^1}=1$, we will also 
consider the so-called \emph{mixed-state localization operators} defined as $A_{\Omega}^\sigma=\chi_\Omega\star \sigma$ \cite{Bible2}. If we write $\sigma=\w\w^\ast$, then 
it was shown in 
 \cite[Section~4.3]{OpSTFT} that  the {mixed-state localization operator } with respect to $\sigma$ is unitarily equivalent to $T_\Omega$, the Toeplitz operator on the quantum Gabor space \(\V_{\w}(\S^2)\) with symbol \(\chi_{\Omega}\): 
    $$
    \chi_\Omega \star( \w\w^\ast)=\V_\w^\ast  {T}_\Omega \V_\w,
    $$
    where $T_\Omega:\V_\w(\S^2)\to\V_\w(\S^2),$ $T_\Omega F=\V_\w\V_\w^\ast(\chi_\Omega\cdot F).$
In particular, any bound on the first eigenvalue of the Toeplitz operator $T_\Omega$ provides a bound for $\chi_\Omega \star( \w\w^\ast)$ as well.  This corresponds to $p=2$ in our main results, see Remark~\ref{rem:mixed-state}. 

\subsection{Modulation spaces of operators}
We define the  class of \textit{admissible operators} as the following subset of the Hilbert-Schmidt operators \begin{align*}
        \mathfrak{M}^1=\left\{\w\in \S^2\colon \mathfrak{V}_\w\w\in L^1(\R^{2d},\S^2)\right\}.
    \end{align*}
\(\mathfrak{M}^1\) is a special case of the \(\mathfrak{A}_v\)-classes which were introduced in \cite{OpSTFT} for the study of co-orbit spaces of operators. \(\A^1\) contains all operators that are nuclear from \(L^2(\R^d)\) to \(M^1(\R^d),\) and is also closed under linear combinations and convolutions with functions in \(L^1(\R^{2d})\). For our purposes, it turns out that \(\mathfrak{M}^1\) is the correct operator analogue of Feichtinger's algebra \(M^1(\R^d)\).

Keeping this analogy in mind, we now define the operator co-orbit spaces \(\M^p\)~\cite{OpSTFT,DKLMcN}. They will be the operator analogues of the modulation spaces $M^p(\R^d)$.
\begin{definition}
    Let \(p\in[1,\infty]\), \(h_0\) be the standard \(d\)-dimensional Gaussian, and let \(\w_0=h_0\otimes h_0\). We define the space \(\A^p\) by
    \begin{align*}    \A^p=\left\{\r\in\mathfrak{S}'\colon \V_{\w_0}\r\in L^p(\R^{2d},\S^2)\right\},
    \end{align*}
    where $\mathfrak{S}'$ denotes the tempered operators, i.e., the operators with Weyl symbols in the space of tempered distributions.

\end{definition}
The spaces $\A^p$ are Banach spaces, $\A^2=\S^2$,  and the inversion formula (Proposition~\ref{prop:op-stft} (3)) extends to these spaces, see  \cite[Section~5]{OpSTFT} for proofs and a detailed discussion. In addition, the rank-one operator \(f\otimes g\) belongs to \(\A^p\) if and only if $f\in M^p(\R^d),$ and $g \in L^2(\R^d)$, so we can consider the classical modulation space \(M^p(\R^d)\) as embedded inside \(\A^p\). Furthermore, \(\A^p\) inherits many of the useful properties of \(M^p(\R^d)\). For instance, replacing the operator \(\w_0\) with any other admissible operator defines an equivalent norm on \(\A^p\). 
\begin{prop}[Special case of Proposition 5.13 in \cite{OpSTFT}]\label{prop:eqNorm}
    Any operator \(\w\in\A^1\) defines an equivalent norm on \(\A^p\). That is, \begin{align*}        
\|\V_\w\r\|_{L^p(\R^{2d},\S^2)}=\|\w^*\pi(z)^*\r\|_{L^p(\R^{2d},\S^2)}\asymp\|\r\|_{\A^p}
    \end{align*}
    for every \(p\in[1,\infty]\).
    Furthermore, since \(\|V_{h_0}f\|_{p}=\|\V_{\w_0}(f\otimes h_0)\|_{L^p(\R^{2d},\S^2)},\) any \(\w\in\A^1\) defines an equivalent norm on \(M^p(\R^d)\)\begin{align*}
        \|\V_\w f\|_{L^p(\R^{2d},L^2(\R^d))}=\|\w^*\pi(z)^*f\|_{L^p(\R^{2d},L^2(\R^d))}\asymp\|f\|_{M^p}.
    \end{align*}
\end{prop}

Lastly, we will need that $\A^p$ is an interpolation space when interpolating between $\A^1$ and $\A^\infty$, which can be proven using similar arguments as in   \cite[Theorem~6.8]{QTFA}.
\begin{prop}\label{prop:interpolation}
    Let \(p\in(1,\infty)\) and \(\theta=1-\frac{1}{p}\). Then we have\begin{align*}
        \big[\A^1,\A^\infty\big]_{\theta}=\A^p
    \end{align*}
\end{prop}
\begin{proof}
    By Proposition 3.16 and Lemma 4.3 in \cite{QTFA}, we have \begin{align*}
        \A^p\cong\mathcal{F}W\big(\mathcal{F}L^p(\R^{d},L^2(\R^d)),L^p(\R^d)\big).
    \end{align*}
    Interpolation between \(\A^p\)-spaces thus corresponds to interpolation between suitable Wiener amalgam spaces. As interpolation in Wiener amalgam spaces can be carried out separately in the local and global components, we get \begin{align*}
        \big[W\big(\mathcal{F}L^1(\R^{d},L^2),&L^1(\R^d)\big),W\big(\mathcal{F}L^\infty(\R^{d},L^2(\R^d)),L^\infty(\R^d)\big)\big]_{\theta}\\&=W\big(\big[\mathcal{F}L^1(\R^{d},L^2(\R^d)),\mathcal{F}L^\infty(\R^{d},L^2(\R^d))\big]_{\theta},\big[L^1(\R^d),L^\infty(\R^d)\big]_{\theta}\big)\\
        &=W\big(\mathcal{F}L^p(\R^{d},L^2(\R^d)),L^p(\R^d)\big),
    \end{align*}
    since \(\big[L^1\big(\R^{d},L^2(\R^d)\big),L^\infty\big(\R^{d},L^2(\R^d)\big)\big]_{\theta}=L^p(\R^{d},L^2(\R^d))\), just like in the scalar-valued case \cite{InterpolationSpaces}. The interpolation result for the \(\M^p\) spaces now follows by using the fact that the Fourier transform is an isomorphism between \(\M^p\) and \(W\big(\mathcal{F}L^p(\R^{d},L^2(\R^d)),L^p(\R^d)\big),\) see, for instance, \cite{feichtinger1983modulation}.
\end{proof}

\subsection{Reinhardt domains} In this section, we recall  some basic properties  of Reinhardt domains, for further details, see \cite{SCVBook}. 
Reinhardt domains are generalizations of radially symmetric subsets of the plane. A set \(\Delta\subset\R^2\) is radially symmetric if and only if any rotation leaves \(\Delta\) invariant. In other words, if we let the circle \(\T\) act on \(\R^2\) by \(e^{i\theta}(x,\omega)=(\cos(\theta)x,\sin(\theta)\omega)\), then \(e^{i\theta}(x,\omega)\in \Delta\) for any \(e^{i\theta}\in \T\) and \((x,\omega)\in \Delta\). For \(d>1\) we consider the action of the torus \(\T^d\) on \(\R^{2d}\) given by
\[
e^{i\theta}(x,\omega)=\big(\cos(\theta_1)x_1,\cos(\theta_2)x_2,\dots,\cos(\theta_d)x_d,\sin(\theta_1)\omega_1,\sin(\theta_2)\omega_2,\dots ,\sin(\theta_d)\omega_d\big).
\] 
We say that a set \(\Delta\subset\R^{2d}\) is a \emph{Reinhardt domain} if for any \(e^{i\theta}\in \T^d\) and \((x,\omega)\in \Delta\) we have \(e^{i\theta}(x,\omega)\in \Delta\). In other words, we have \[e^{i\theta}\Delta=\Delta,\qquad \text{ for all } e^{i\theta}\in \T^d.\]
Reinhardt domains have a nice description in terms of polyradial coordinates: $x_i=r_i\cos(\theta_i),$ $\omega_i=r_i\sin(\theta_i)$. It then turns out that any Reinhardt domain \(\Delta\) may be written as \(W\times\T^d\), where \(W\subset\R_{\geq 0}^d\). The set \(W\) is uniquely determined by \(\Delta\) and vice versa, and we will refer to \(W\) as the \emph{Reinhardt shadow} of \(\Delta\) \cite{ReinhardtArticle}.
Many familiar domains are Reinhardt domains. Below we collect some examples and their Reinhardt shadows. \begin{exmp}
    \begin{itemize}
        \item The ball in \(\R^{2d}\) with radius \(R\) is a Reinhardt domain. Its shadow is \(\{r\in \R_{\geq 0}^d\colon \sum_{j=1}^d r_j^2\leq R^2\}\).
        \item The polydisk in \(\R^{2d}\) with polyradius \(R=(R_1,R_2,\dots, R_d)\) is a Reinhardt domain. Its shadow is \(\prod_{j=1}^d [0,R_j]\subset \R_{\geq 0}^d\).
        \item In general, let \(p\in (0,\infty]\). Balls in the \(p\)-(semi)norm centered at the origin are Reinhardt domains. Their shadows are given by \(\{r\in \R_{\geq0}^d\colon \sum_{j=1}^d r_j^p\leq R^p\}\).
    \end{itemize}
\end{exmp}

\section{Eigenvalues of vector-valued time-frequency localization operators}\label{sec:S^2=L^2}

In this section, we  discuss some basic properties of the spectrum of the mixed-state localization operator $A_\Omega^{\w\w^\ast}$ defined on $L^2(\R^d)$ and on $\S^2$ respectively, and show that, for $p=2$, the solutions to the  optimal concentration problems   \eqref{Rayleigh} and \eqref{RayleighOp} on $\S^2$ and  on $L^2(\R^d)$ coincide. In other words, the concentration problem on operators reduces to the same concentration problem on pure states. 

 Recall our standing assumption \(\|\w\|_{\S^2}=1\). 
First, we note that if $\Omega$ is compact, then $A_\Omega^{\w\w^\ast}: L^2(\R^d)\to L^2(\R^d)$ is a trace class operator. To see this consider the singular value decomposition of $\w=\sum_{n\in\N}\lambda_n (g_n\otimes f_n)$ and let $\{e_n\}_{n\in\N}$ be any orthonormal basis of $L^2(\R^d)$. Then 
\begin{align*}
    \text{trace}(A_\Omega^{\w\w^\ast})&=\sum_{m\in\N}\langle A_\Omega^{\w\w^\ast}e_m ,e_m\rangle =\sum_{n,m\in\N}\int_\Omega |\lambda_n|^2 \big\langle \pi(z)(g_n\otimes g_n)\pi(z)^\ast e_m,e_m\big\rangle dz
    \\
    &=\sum_{n,m\in\N}\int_\Omega |\lambda_n|^2\big| \big\langle \pi(z)  g_n,e_m\big\rangle\big|^2 dz=\sum_{n\in\N}\int_\Omega |\lambda_n|^2 dz=|\Omega|.
\end{align*}
If $\w$ satisfies certain additional localization assumptions then it can be shown that the eigenvalue profile of $A_\Omega^{\w\w^\ast}$ behaves similar to the classical time-frequency localization operators. See  \cite{spec-dev-2} for a detailed discussion of that matter.

\begin{prop}\label{prop:vector=operator}
    The mixed-state localization operators $A_\Omega^{\w\w^\ast}: L^2(\R^d)\to L^2(\R^d)$ and $A_\Omega^{\w\w^\ast}: \S^2\to  \S^2$ have the same eigenvalues. Moreover, every eigenvalue of $A_\Omega^{\w\w^\ast}:\S^2\to  \S^2$ has infinite multiplicity and the corresponding eigenvectors have rank one
    unless the corresponding eigenvalue of $A_\Omega^{\w\w^\ast}: L^2(\R^d)\to L^2(\R^d)$ has multiplicity greater than one. Finally,
    \begin{align}\label{eq:equal-norm}
    \sup_{\r\in\S^2} \frac{\int_\Omega \|\V_\w\r(z)\|_{\S^2}^2dz}{\|\r\|_{\S^2}^2}=\sup_{f\in L^2(\R^d)} \frac{\int_\Omega \|\V_\w f(z)\|_{2}^2dz}{\|f\|_{2}^2},
\end{align}
and   the left-hand side is maximized by a rank-one operator.
\end{prop}

\begin{remark}
    Note that    $A_\Omega^{\w\w^\ast}:\S^2\to\S^2$ is not compact since every nonzero eigenvalue has infinite multiplicity. The equality \eqref{eq:equal-norm} is therefore not a mere consequence of the min-max-theorem and   the 
observation that the point spectra of $A_\Omega^{\w\w^\ast}:L^2(\R^d)\to L^2(\R^d)$ and $A_\Omega^{\w\w^\ast}:\S^2\to\S^2$  are equal.
\end{remark}
\begin{proof}[Proof of Proposition~\ref{prop:vector=operator}]
    Let $\{\lambda_n\}_{n\in\N}$ be the collection of nonzero eigenvalues of $A_\Omega^{\w\w^\ast}:L^2( \R^d)\to L^2( \R^d) $  in non-increasing order, and   $\{f_n\}_{n\in\N}\subset L^2(\R^d)$ be an orthonormal family such that 
    $
    A_\Omega^{\w\w^\ast}  f_n=\lambda_n  f_n.
    $
  Then for every $g\in L^2(\R^d):$
    $$
    A_\Omega^{\w\w^\ast}  (f_n\otimes g)=   \big(A_\Omega^{\w\w^\ast} f_n\big)\otimes g = (\lambda_n  f_n)\otimes g=\lambda_n (f_n\otimes g).
    $$
    In other words, $f_n\otimes g$ is an eigenvector  in $\S^2$ for the eigenvalue $\lambda_n$. Since $g$ is arbitrary, it clearly follows that the multiplicity of $\lambda_n$ is infinite for every $n\in\N$.

    If $\{g_n\}_{n\in\N}$ is an orthonormal basis of $\overline{\text{span}}^\bot(\{f_n\}_{n\in\N})$ and $\{e_k\}_{k\in\N}$ is an orthonormal basis of $L^2(\R^d)$, then we may write every $\r\in \S^2$ as
\begin{equation}\label{eq:T-in-HS-basis}
\r=\sum_{k,n\in\N} \alpha_{n,k} (f_n\otimes e_k)+\sum_{k,n\in\N} \beta_{n,k} (g_n\otimes e_k),
\end{equation}
where $\{\alpha_{n,k}\}_{(n,k)\in\N^2},\{\beta_{n,k}\}_{(n,k)\in\N^2}\in \ell^2(\N^2).$ In particular, $\{g_n\}_{n\in\N}$ is is a basis of  $\text{ker}(A_\Omega^{\w\w^\ast})$.
Now assume that $\r$ is an eigenvector of $A_\Omega^{\w\w^\ast}$ with corresponding  eigenvalue $\nu\neq 0$. Then
    $$
  \nu \r=  A_\Omega^{\w\w^\ast} \r =\sum_{k,n\in\N} \alpha_{n,k} (A_\Omega^{\w\w^\ast} f_n\otimes e_k)+\sum_{k,n\in\N} \beta_{n,k} (A_\Omega^{\w\w^\ast} g_n\otimes e_k)=\sum_{k,n\in\N} \lambda_n \alpha_{n,k}  (f_n\otimes e_k),
    $$
which is only possible  if $\beta_{n,k}=0$   and
$\alpha_{n,k}=\frac{\lambda_n}{\nu}\alpha_{n,k} $ for every $(n,k)\in\N^2$.  In particular, $\nu=\lambda_{n^\ast}$  for some $n^\ast\in\N$ and $\{\alpha_{n,k}\}_{k\in\N}=0$ if $\lambda_n\neq \lambda_{n^\ast}$. If $\lambda_{n^\ast}$ has multiplicity one, then it follows that $\text{rank}(\r)=1.$

To prove \eqref{eq:equal-norm} we note that  from \eqref{eq:T-in-HS-basis}
 \begin{align*}
    \sup_{\r\in\S^2} \frac{\int_\Omega \|\V_\w\r(z)\|_{\S^2}^2dz}{\|\r\|_{\S^2}^2}&=\sup_{\r\in\S^2} \frac{\langle A_\Omega^{\w\w^\ast} \r,\r\rangle_{\S^2}}{\|\r\|_{\S^2}^2}\\
    &=\sup_{\alpha,\beta\in\ell^2(\N^2)} \frac{\sum_{n,m\in\N}|\alpha_{n,k}|^2\langle A_\Omega^{\w\w^\ast} f_n,f_n\rangle +|\beta_{n,k}|^2\langle A_\Omega^{\w\w^\ast}g_n,g_n\rangle}{\|\alpha\|_{\ell^2}^2+\|\beta\|_{\ell^2}^2}
    \\
    &=\sup_{f\in L^2(\R^d)} \frac{\langle A_\Omega^{\w\w^\ast} f,f\rangle}{\|f\|_{2}^2} =\sup_{f\in L^2(\R^d)} \frac{\int_\Omega \|\V_\w f(z)\|_{2}^2dz}{\|f\|_{2}^2}.
\end{align*} 

 \end{proof}

\begin{remark}
 Equation \eqref{eq:equal-norm}  has interesting consequences when one considers the problem of optimal concentration among all concentration domains of  fixed measure. As already discussed in the introduction,     Nicola and Tilli showed in their landmark paper \cite{faber-krahn} that the following Faber-Krahn inequality holds for the STFT with Gaussian window $h_0$ 
 \begin{equation}\label{eq:Faber-Krahn}
 \sup_{   { |\Omega|=C}}\ \sup_{f\in L^2(\R)} \frac{\int_\Omega |V_{h_0} f(z)|^2dz}{\|f\|_{2}^2}\leq 1-e^{-C},\qquad \Omega\subset \R^2,
 \end{equation}
 with equality occurring if and only if (up to measure zero perturbations) $\Omega=D_{\sqrt{C/\pi} }(z^\ast)$ and $f=\pi(z^\ast)h_0$ for some $z^\ast\in\R^2$.
Note that  a stable version of this inequality was introduced in  \cite{stable-faber-krahn}. 
 If we consider $\w=h_0\otimes g$ for an arbitrary $g\in L^2(\R)$ with $\|g\|_2=1$, then we deduce from a  combination of \eqref{eq:equal-norm} and \eqref{eq:Faber-Krahn}  that 
 $$
 \sup_{ |\Omega|=C}\sup_{\r\in\S^2} \frac{\int_\Omega \|\V_\w\r(z)\|_{\S^2}^2dz}{\|\r\|_{\S^2}^2}\leq 1-e^{-C}, \qquad \Omega\subset \R^2,
 $$
 where equality holds if and only if    one has $\Omega=D_{\sqrt{C/\pi} }(z^\ast)$ (up to sets of measure zero)   and $\r=(\pi(z^\ast)h_0)\otimes f$   for some $z^\ast\in\R^2$ and $f\in L^2(\R)$ with $\|f\|_2=1$.
\end{remark}

\section{A large sieve principle for the operator STFT}\label{sec:large-sieve}

We now prove an operator-valued version of the large sieve principle presented in \cite{LargeSieve}. We will start by proving the result in the case where both \(\w\) and \(\r\) are admissible. The result can then be extended to \(\r\in \A^p\) by interpolation.

\begin{prop}\label{prop:abstract-LS}
    Let \(\w,\r\in \A^1\) with \(\|\w\|_{\S^2}=1\) and let \(\Omega\subset\R^{2d}\) be measurable. 
    If $\mathcal{K}:\R^{2d}\times\R^{2d}\to \mathcal{B}(L^2(\R^d))$ is such that 
   \begin{equation}\label{eq:ass-K}
I_\mathcal{K}(\V_\w\r):= \int_{\R^{2d}}\mathcal{K}(\, \cdot\, ,w)\V_\w\r(w)dw 
    \end{equation}
 is bounded and boundedly invertible on $\V_\w(\A^1)\subset L^1(\R^{2d},\S^2)$, then we have 
    \begin{align*}
        \frac{\|\chi_{\Omega}\cdot\V_\w\r\|_{L^1(\R^{2d},\S^2)}}{\|\V_\w\r\|_{L^1(\R^{2d},\S^2)}}\leq \theta_\mathcal{K}\cdot\sup_{w\in \R^{2d}}\int_\Omega \left\|\mathcal{K}(z,w)\right\|_{\mathrm{op}}dz,
    \end{align*} 
    with 
    $$
\theta_\mathcal{K}:=\sup_{\r\in\A^1}\left(\frac{\|\V_\w\r\|_{L^1(\R^{2d},\S^2)}}{\|I_\mathcal{K}(\V_\w\r)\|_{L^1(\R^{2d},\S^2)}}\right).
    $$
\end{prop}

\begin{proof}
    We adapt the proof of  \cite[Proposition~1]{LargeSieve} to our operator-valued setting. The assumption on $\mathcal{K}$ implies that for each $\r\in\A^1$ there exists a unique $\widetilde{\r}\in\A^1$ such that 
    $$
   \V_\w\r(z)= \int_{\R^{2d}}\mathcal{K}(z   ,w)\V_\w \widetilde{\r}(w)dw 
    $$
    for almost every $z\in\R^{2d}.$
    Therefore,
    \begin{align*}
\int_{\Omega}\left\|\V_\w\r(z)\right\|_{\S^2}dz 
        &=\int_{\Omega}\left\|\int_{\R^{2d}}\mathcal{K} (z,w)\mathfrak{V}_\w\widetilde{\r}(w)dw\right\|_{\S^2}dz
        \\
        &\leq \int_{\Omega}\int_{\R^{2d}}\left\|\mathcal{K}(z,w)\mathfrak{V}_\w\widetilde{\r}(w)\right\|_{\S^2}dwdz 
        \\
        &\leq \int_{\R^{2d}}\int_{\Omega}\left\|\mathcal{K}(z,w)\right\|_{\mathrm{op}}\big\|\mathfrak{V}_\w\widetilde{\r}(w)\big\|_{\S^2}dzdw
        \\
        &\leq \sup_{w\in\R^{2d}}\int_{\Omega}\left\|\mathcal{K}(z,w)\right\|_{\mathrm{op}}dz \cdot \int_{\R^{2d}} \big\|\mathfrak{V}_\w\widetilde{\r}(w)\big\|_{\S^2}dw
        \\
        &=\frac{\big\|\mathfrak{V}_\w\widetilde{\r}\big\|_{L^1(\R^{2d},\S^2)}}{ \|\mathfrak{V}_\w\r \|_{L^1(\R^{2d},\S^2)}}\cdot \sup_{w\in\R^{2d}}\int_{\Omega}\left\|\mathcal{K}(z,w)\right\|_{\mathrm{op}}dz \cdot \big\|\mathfrak{V}_\w {\r}\big\|_{L^1(\R^{2d},\S^2)}
        \\ &\leq \theta_\mathcal{K}\cdot\sup_{w\in\R^{2d}}\int_{\Omega}\left\|\mathcal{K}(z,w)\right\|_{\mathrm{op}}dz \cdot \big\|\mathfrak{V}_\w {\r}\big\|_{L^1(\R^{2d},\S^2)}.
    \end{align*}
    The second inequality is due to the ideal property of \(\S^2\), i.e., \(\|AB\|_{\S^2}\leq \|A\|_{\mathrm{op}}\|B\|_{\S^2}\).
\end{proof}

We now state the general large sieve principle for the operator STFT.

\begin{prop}\label{QuantumSieve}
Let \(p\in[1,\infty)\), \(\w\in \A^1\) with \(\|\w\|_{\S^2}=1\) and let \(\Omega\subset\R^{2d}\) be measurable. If $\mathcal{K}$ is as in Proposition~\ref{prop:abstract-LS}, then we have for \(\r\in \A^p\)
\begin{align*}
        \frac{\|\chi_{\Omega}\cdot\V_\w\r\|_{L^p(\R^{2d},\S^2)}^p}{\|\V_\w\r\|_{L^p(\R^{2d},\S^2)}^p}\leq \theta_\mathcal{K} \cdot\sup_{w\in \R^{2d}}\int_\Omega \left\|\mathcal{K}(z,w)\right\|_{\mathrm{op}} dz.
\end{align*} 
\end{prop}
\begin{proof}
    Proposition~\ref{prop:abstract-LS} proves the case \(p=1\).  For \(p=\infty\) we clearly have
    \begin{align*}
        \sup_{z\in\R^{2d}} \|\chi_\Omega(z)\cdot\V_\w\r(z)\|_{\S^2}=\sup_{z\in\R^{2d}}\chi_\Omega(z)\cdot \|\V_\w\r(z)\|_{\S^2}\leq \sup_{z\in\R^{2d}} \|\V_\w\r(z)\|_{\S^2}.
    \end{align*} Thus we also have an estimate for \(\A^\infty.\) Now we apply complex interpolation (Proposition~\ref{prop:interpolation}) to get the result.
\end{proof}

\begin{remark}
    This proof strategy also works  when the characteristic function \(\chi_{\Omega}\) is replaced by a general weight function \(F\in L^{\infty}(\R^{2d})\). In this case, however, one picks up an additional constant on the right-hand side, since $$\sup_{z\in\R^{2d}}\|F(z)\V_\w\r(z)\|_{\S^2}\leq \|F\|_{\infty}\cdot\sup_{z\in\R^{2d}}\|\V_\w\r(z)\|_{\S^2}.$$ In particular, the correct inequality is as follows \begin{align*}
                \frac{\|F\cdot\V_\w\r\|_{L^p(\R^{2d},\S^2)}^p}{\|\V_\w\r\|_{L^p(\R^{2d},\S^2)}^p}\leq \theta_\mathcal{K}\cdot\|F\|_{\infty}^{p-1}\cdot\sup_{w\in \R^{2d}}\int_{\R^{2d}} \left\|\mathcal{K}(z,w)\right\|_{\mathrm{op}}F(z)dz.
    \end{align*}
\end{remark}
\begin{remark}\label{rem:mixed-state}
From the observation \eqref{eq:norm-S^2=norm-L^2}, we infer that Proposition~\ref{QuantumSieve} also provides the following estimate for $f\in M^p(\R^d)$
\begin{align*}
        \frac{\|\chi_{\Omega}\cdot\V_\w f\|_{L^p(\R^{2d},L^2(\R^d))}^p}{\|\V_\w f\|_{L^p(\R^{2d},L^2(\R^d))}^p}\leq \theta_\mathcal{K} \cdot\sup_{w\in \R^{2d}}\int_\Omega \left\|\mathcal{K}(z,w)\right\|_{\mathrm{op}} dz.
\end{align*} 
Finally, the discussion in Section~\ref{sec:S^2=L^2} shows that  the first eigenvalue of the
mixed-state localization operator  $A_\Omega^{\w\w^\ast}:L^2(\R^d)\to L^2(\R^d)$ satisfies 
$$
\lambda_1\big(A_\Omega^{\w\w^\ast}\big)\leq \theta_\mathcal{K} \cdot\sup_{w\in \R^{2d}}\int_\Omega \left\|\mathcal{K}(z,w)\right\|_{\mathrm{op}} dz.
$$
This estimate holds for quite general sets \(\Omega\), operators \(\w\) and kernels \(\mathcal{K}\) and as such, it is not optimal.
\end{remark}
So far, our large sieve estimate is  rather abstract. The following section  will therefore be concerned with studying  different kernels $\mathcal{K}$ and deriving concrete estimates of the corresponding constant $\theta_\mathcal{K}$.

\section{Estimating the constant \texorpdfstring{$\theta_\mathcal{K}$}{theta-K}}\label{sec:estimating-theta}

\subsection{Proof of Theorem \ref{thm:2}}\label{sec:reproducing-kernel}

We first note that, if we use the reproducing kernel $K_\w$ as the integral kernel, we immediately derive the following result.  
\begin{corollary}
    \label{QuantumSieve_1}
Let \(p\in[1,\infty)\), \(\w\in \A^1\) with \(\|\w\|_{\S^2}=1\) and let \(\Omega\subset\R^{2d}\) be measurable. For \(\r\in \A^p\), we have\begin{align*}
        \frac{\|\chi_{\Omega}\cdot\V_\w\r\|_{L^p(\R^{2d},\S^2)}^p}{\|\V_\w\r\|_{L^p(\R^{2d},\S^2)}^p}\leq \sup_{w\in \R^{2d}}\int_\Omega \left\|\w^*\pi(z)^*\pi(w)\w\right\|_{\mathrm{op}}dz.
\end{align*} 
\end{corollary}
 \begin{proof}
     If we use $\mathcal{K}=K_\w$, then $\theta_\mathcal{K}=1$ and the statement follows from Proposition~\ref{QuantumSieve}.
 \end{proof}
Now, observe that if the operator window is a rank-one operator \(g\otimes g\) with \(g\in M^1\), then the operator norm of the kernel simplifies to the absolute value of the kernel of the scalar valued STFT, i.e.,
\begin{align*}
        \frac{\|\chi_{\Omega}\cdot\V_{g\otimes g}\r\|_{L^p(\R^{2d},\S^2)}^p}{{\|\V_{g\otimes g}}\r\|_{L^p(\R^{2d},\S^2)}^p}\leq \sup_{w\in \R^{2d}}\int_\Omega |\langle \pi(w)g,\pi(z)g\rangle|dz=\sup_{w\in \R^{2d}}\int_\Omega |K_g(z,w)|dz.
\end{align*}
We note that although the operator norm of the kernel gives the best estimate, it is   in many situations   easier to estimate the right-hand side of Proposition~\ref{QuantumSieve} by a Schatten norm since one can exploit the trace properties to get results that are easier to interpret or are even explicit. In this case, estimating the right-hand side using the Hilbert-Schmidt norm, and writing \(\sigma=\w\w^*\), we have by Lemma~\ref{Identity}
\begin{equation}
  \|\w^*\pi(z)^*\pi(w)\w\|_{\mathrm{op}}\leq   \|\w^*\pi(z)^*\pi(w)\w\|_{\S^2}=\|\V_\w\w(z-w)\|_{\S_2}=\sqrt{\widecheck{\sigma}\star  {\sigma}(z-w)},\label{eq:HS=conv}
\end{equation}
an estimate that we will leverage in the following example, which completes the proof of Theorem \ref{thm:2}.

\begin{exmp}
Let $d=1$ and  consider  the Weyl transform  $L_{g_a}$ of the Gaussian  \(g_a\), $a>0$, defined by \(g_a(t)=\frac{2}{1+2a}e^{-\frac{2\pi}{1+2a}|t|^2},\ t\in\R^2,\) which has the spectral decomposition \begin{align*}
    L_{g_a}=\frac{1}{a+1}\sum_{n\in\N_0} \left(\frac{a}{a+1}\right)^n (h_n\otimes h_n).
\end{align*} 
Note that the normalization of $g_a$ is chosen such that $\|g_a\|_1=\|L_{g_a}\|_{\S^1}=1. $
    Now take  \(\w=\sqrt{L_{g_a}}=\frac{1}{\sqrt{1+a}}\sum_{n\in \N_0}\big(\frac{a}{a+1}\big)^{n/2}(h_n\otimes h_n)\), so that \(\w\w^*=L_{g_a}\) gives us the Cohen's classes \(W(f,f)*g_a\), which are popular choices in the literature, due to their positivity \cite{Grochenig,deBruijn,Yvon}.
Instead of trying to compute the operator norm of the expansion \eqref{eq:kernel-written-out} we instead use \eqref{eq:HS=conv}. Together with \eqref{eq:convolution-op-vs-symbol} and the semigroup property of the Gaussian  this yields the following nice explicit expression for the Hilbert-Schmidt norm of the reproducing kernel
\begin{align*}
    \|\w^*\pi(z)^*\pi(w)\w\|_{\S^2}&=\sqrt{\widecheck{L_{g_a}}\star  {L_{g_a}}(z-w)}=\sqrt{L_{g_a}\star L_{g_a}(z-w)}\\
    &=\sqrt{g_a*g_a(z-w)}=
    \frac{1}{ \sqrt{1+2a}}e^{-\frac{\pi}{2(1+2a)} |z-w|^2}.
    \end{align*}
    Setting $w=0$ above, we see that $\V_\gamma\gamma\in L^1(\R^2,\S^2)$, that is $\gamma$ is admissible and we may apply Proposition~\ref{prop:abstract-LS}
to derive  the estimate
\begin{align*}
        \frac{\|\chi_{\Omega}\cdot\V_\w\r\|_{L^p(\R^{2},\S^2)}^p}{\|\V_\w\r\|_{L^p(\R^{2d},\S^2)}^p}\leq \frac{1}{\sqrt{1+2a}} \cdot\sup_{w\in \R^{2}}\int_\Omega  e^{-\frac{\pi}{2(1+2a)} |z-w|^2} dz\leq {2\sqrt{1+2a} \left(1-e^{-\frac{|\Omega |}{ {2(1+2a)}}}\right)},
\end{align*} 
where in the last step we used that the Gaussian integral is maximized for $\Omega$ being a ball centered at $w$. In other words, we have established Theorem~\ref{thm:2}.
\end{exmp}

\subsection{Local reproducing formulas and proof of Theorem~\ref{thm:1}}\label{LocalRep}

In order to interpret Proposition~\ref{QuantumSieve} as a measure of the sparsity of \(\Omega\) we will derive and employ a generalization of the local reproducing formula announced in the introduction (which holds for polyradial operators and Reinhardt domains) for the operator STFT. 
An ingredient of the proof is the following result from \cite[Lemma~5.2]{svela}, which generalized \cite[Proposition~6]{LargeSieve} to higher dimensions.
\begin{lemma}\label{MyDblOrth}
    Let $W\subset \R_{\geq 0}^d$ be bounded,  let $\Delta\subset\R^{2d}$ be a Reinhardt domain with Reinhardt shadow $W$, and let $m\in\N_0^d$. Then, the short-time Fourier transforms of the Hermite functions \(\{V_{h_m}h_n\}_{n\in\N_0^d}\) are orthogonal with respect to $\chi_{\Delta}(z)dz$, i.e., 
    \begin{align}    \int_{\Delta}V_{h_m}h_n(z)\overline{V_{h_m}h_k(z)}dz=C_{n,m}(\Delta)\cdot\delta_{n,k}\qquad n,k\in\N_0^d,
    \end{align}
    where \begin{align}\label{eq:def-C_nm}
       C_{n,m}(\Delta)=(2\pi)^d\frac{m!}{n!}\pi^{|n-m|}\int_W r^{2n-2m+1}e^{-\pi |r|^2} \left(\prod_{j=1}^dL_{m_j}^{n_j-m_j}(\pi r_j^2)\right)^2 dr.
    \end{align}
\end{lemma}
The above lemma shows that the following identity holds weakly
$$
\int_{\Delta}\langle  h_n,\pi(w)h_m\rangle\pi(w){h_m}dw=C_{n,m}(\Delta) h _n.
$$
If we then argue as in \cite[Theorem~1]{LargeSieve}, we deduce that
\begin{equation}\label{eq:rep-form-weak}
    \int_{z+\Delta}\langle  \pi(z)h_n,\pi(w)h_m\rangle\pi(w){h_m}dw=C_{n,m}(\Delta) \cdot\pi(z)h_n
\end{equation}
 holds weakly. Therefore,
  we derive the following local reproducing formula
\begin{equation}
V_{h_n}f(z)=   \frac{1}{C_{n,m}(\Delta)} \int_{z+\Delta}V_{h_m}f(w)\langle  \pi(w)h_m,\pi(z)h_n\rangle dw   ,\qquad z\in\R^{2d}.
\end{equation}
Note that here, the value of the STFT with window $h_n$ at a point $z$ is recovered from the values of the STFT with window $h_m$ in a neighborhood of $z$. Compare this to \cite[Theorem~1]{LargeSieve} where the $V_{h_n}f(z)$ is recovered from the local information on  $V_{h_n}f$ around $z$.

We are now ready to state our operator analogue of the local reproducing formula for the STFT. Unlike the classical case, where we reproduce the STFT precisely, the quantum local reproducing formula yields the value of an operator STFT of the same signal, but with a different operator window. For the classical STFT, the Hermite functions are the particular windows that admit a local reproducing formula. In the operator case, the Hermite functions are replaced by the class of \textit{polyradial operators}, that is, the operators that admit a diagonal decomposition in terms of the Hermite basis: \(\w=\sum_{n\in \N_0^d}\lambda_n (h_n\otimes h_n)\).
\begin{prop}[Quantum local reproducing formula]\label{prop:quant-loc-rep}
    Let \(\Delta\subset\R^{2d}\) be a Reinhardt domain, and let \(\w=\sum_{n\in \N_0^d}\lambda_n( h_n\otimes h_n)\) be a polyradial operator with \(\|\w\|_{\S^2}=1\). 
    For any  \(z\in \R^{2d}\), we have
    \begin{align}\label{eq:weak-loc-rep}
        \int_{z+\Delta}K_\w(z,w)\w^\ast\pi(w)^\ast dw= \big(\widetilde{\w}(\Delta)\big)^\ast \pi(z)^\ast,
    \end{align}
    where 
    $$
    \widetilde{\w}(\Delta)=\sum_{m\in \N_0^d} \lambda_m A_m(\Delta)(h_m\otimes h_m),
    $$
    and
    \begin{align*}
        A_m(\Delta)=\sum_{n\in \N_0^d}|\lambda_n|^2 C_{m,n}(\Delta),
    \end{align*} and \(C_{m,n}(\Delta)\) is given by \eqref{eq:def-C_nm}. In particular, for every $\r\in \S^2$  
    one has
    \begin{equation}\label{eq:loc-rep}
  \V_{\widetilde{\w}(\Delta)}\r(z)=   \int_{z+\Delta}K_\w(z,w)\V_\w\r(w)dw.
    \end{equation}
\end{prop}
\begin{proof}
   Let us first write out $K_\w(z,w)\w^\ast\pi(w)^\ast$ explicitly:
   \begin{align*}
       K_\w(z,w)\w^\ast\pi(w)^\ast&=\w^*\pi(z)^\ast\pi(w)\w\w^{\ast}\pi(w)^\ast
      \\ & =\w^*\pi(z)^\ast\pi(w)\left(\sum_{n\in \N_0^d} |\lambda_n|^2 h_n\otimes( \pi(w)h_n) \right)
       \\
       &=\w^* \left(\sum_{n\in \N_0^d} |\lambda_n|^2 \big(\pi(z)^\ast\pi(w)h_n\big)\otimes \big(\pi(w)h_n\big) \right)
       \\
       &=\sum_{m,n\in \N_0^d} \overline{\lambda_m} |\lambda_n|^2 \langle \pi(w)h_n,\pi(z)h_m\rangle\big( h_m\otimes (\pi(w)h_n) \big)
     \\     &=\sum_{m\in \N_0^d} \overline{\lambda_m}\, h_m\otimes \left(\sum_{n\in \N_0^d}  |\lambda_n|^2 \langle \pi(z)h_m,\pi(w)h_n\rangle\, \pi(w)h_n  \right).
   \end{align*}
   The last equality holds due to our convention of the tensor product being antilinear in the second argument. Integrating over $z+\Delta$ thus yields by \eqref{eq:rep-form-weak}
      \begin{align*}
     \int_{z+\Delta}  K_\w(z,w)\w^\ast\pi(w)^\ast dw&= 
  \sum_{m\in \N_0^d} \overline{\lambda_m}\, h_m\otimes \left( \sum_{n\in \N_0^d}   |\lambda_n|^2 \int_{z+\Delta}\langle \pi(z)h_m,\pi(w)h_n\rangle\, \pi(w)h_n  dw\right)
  \\
  &= 
 \sum_{m\in \N_0^d} \overline{\lambda_m}\, h_m\otimes \left( \sum_{n\in \N_0^d}   |\lambda_n|^2 C_{m,n}(\Delta)\pi(z)h_m \right)
    \\
  &= 
  \sum_{m\in \N_0^d} \overline{\lambda_m} A_m(\Delta)\,h_m\otimes  \pi(z)h_m  
  =\big(\widetilde{\w}(\Delta)\big)^\ast\pi(z)^\ast.
   \end{align*}
   The identity \eqref{eq:loc-rep} follows if we multiply both sides of \eqref{eq:weak-loc-rep} from the right by $\r.$
\end{proof}
\begin{remark}
    While the operator window \(\widetilde{\w}(\Delta)\) is in general neither equal to   nor a scalar multiple of \(\w\), it is still a polyradial operator. In the special case where $\w=h_n\otimes h_n$  is a rank-one operator, then $A_m(\Delta)=\delta_{n,m}C_{m,n}(\Delta)$ and $\widetilde{\w}(\Delta)=C_{n,n}(\Delta)\w$. This leads to a local reproducing formula reminiscent of the original version obtained in   \cite{LargeSieve} 
    \begin{align}\label{local-rep-rank1}
            \V_{h_n\otimes h_n}\r(z)=\frac{1}{C_{n,n}(\Delta)}\int_{z+\Delta}\langle \pi(w)h_n,\pi(z)h_n\rangle   \V_{h_n\otimes h_n}\r(w)dw.
    \end{align}
    In this case, the constant \(A_m(\Delta)\) is the  eigenvalue of the classical localization operator \(A_{\Delta}^{h_n}=\chi_{\Delta}\star h_n\otimes h_n\) corresponding to the \(m\)-th Hermite function. In general, \(A_m(\Delta)\) coincides with the eigenvalue of the mixed-state localization operator \(A_{\Delta}^{\w\w^*}=\chi_{\Delta}\star \w\w^*\) that  corresponds to the \(m\)-th Hermite function.
\end{remark}
In order to apply Proposition~\ref{QuantumSieve} we need to compute an upper bound of 
\begin{equation}\label{eq:theta_k-polyradial}
\theta_{\mathcal{K}}=\sup_{\r\in\A^1}\frac{\|\V_{{\w}} \r\|_{L^1(\R^{2d},\S^2)}}{\|\V_{\widetilde{\w}(\Delta)} \r\|_{L^1(\R^{2d},\S^2)}}
\end{equation} 
where the denominator can be written as
\begin{align*}
   \|\V_{\widetilde{\w}(\Delta)} \r\|_{L^1(\R^{2d},\S^2)} &=\int_{\R^{2d}}\left\|\sum_{m\in \N_0^d} \overline{\lambda_m}A_m(\Delta)\V_{h_m\otimes h_m}\r(z) \right\|_{\S^2} dz 
   \\
   &=\int_{\R^{2d}}\left(\sum_{m\in \N_0^d} |\lambda_m|^2A_m(\Delta)^2\|\V_{h_m\otimes h_m}\r(z)\|_{\S^2}^2\right)^{1/2}dz.
\end{align*}
There is no one-fits-all  recipe for how to bound \eqref{eq:theta_k-polyradial} as the spectral structure of $S$ critically influences the value of $\theta_\mathcal{K}$. Nevertheless, we present several strategies that may be useful.
First, for
\begin{equation}
B(\Delta)=\inf\big\{A_m(\Delta):\ m\in\N_0^d,\ \lambda_m A_m(\Delta)\neq 0\big\},
\end{equation}
one deduces  that 
\begin{align*}
 \|\V_{\widetilde{\w}(\Delta)} \r\|_{L^1(\R^{2d},\S^2)} 
 &\geq B(\Delta)\cdot\int_{\R^{2d}}\left(\sum_{m\in \N_0^d}|\lambda_m|^2\|\V_{h_m\otimes  h_m}\r(z)\|_{\S^2}^2\right)^{1/2}dz \\&=B(\Delta)\cdot\|\V_{{\w}} \r\|_{L^1(\R^{2d},\S^2)},
\end{align*}
and consequently 
\begin{align}\label{eq:theta-B}
\theta_\mathcal{K}\leq B(\Delta)^{-1}.
\end{align}
Note that in general $B(\Delta)$ might be zero.  If $\w$ is a finite rank polyradial operator, however, then $B(\Delta)>0$ is guaranteed.
In the following, we  consider particular examples of how to lower bound $B(\Delta)$  and discuss cases when \eqref{eq:theta-B} does not provide  a good estimate.

\begin{exmp}
Let  $d=1$ and $\Delta=D_R(0)$.  The simplest example of a polyradial operator with rank greater than one is the rank-two operator \(\w=\frac{1}{\sqrt{2}}(h_0\otimes h_0+h_1\otimes h_1)\). In this setting we can compute
\begin{align*}
B(D_R(0))&=\min_{m\in\{0,1\}}A_m(D_R(0))\end{align*}
explicitly.
Indeed, 
\begin{align*}
    A_0(D_R(0))&=\frac{1}{2}(C_{0,0}+C_{0,1})=\frac{1}{2}\int_0^{\pi R^2}  e^{-t} dt+ \frac{1}{2} \int_0^{\pi R^2} t e^{-t}   dt =1-\frac{2+\pi R^2  }{2}e^{-\pi R^2},
\end{align*}
while \begin{align*}
    A_1(D_R(0))&=\frac{1}{2}(C_{1,0}+C_{1,1})= \frac{1}{2}\int_0^{\pi R^2} te^{-t} dt+ \frac{1}{2} \int_0^{\pi R^2} \big(1-t\big) ^2 e^{-t}   dt \\ &=1-\frac{2+\pi R^2+\pi^2 R^4}{2} e^{-\pi R^2} .
\end{align*}
Since \(A_0(D_R(0))-A_1(D_R(0))=\frac{\pi^2 R^4}{2} e^{-\pi R^2} >0\) for all \(R>0\) we conclude that $$B(D_R(0))=A_1(D_R(0))=1-\frac{2+\pi R^2+\pi^2 R^4}{2} e^{-\pi R^2} .$$

\end{exmp}

\begin{exmp}\label{ex:thm1}  The estimate  for $\theta_\mathcal{K}$ that we will subsequently obtain is used to  prove Theorem~\ref{thm:1} at the end of this section.
Again assume that $d=1$, $\Delta=D_R(0)$, and consider $\w=\sum_{n=0}^N\lambda_n (h_n\otimes h_n)$, for some fixed $N\in \N$. In this example, we show how to lower bound $B(D_R(0))$ when $R$ is chosen according to the rank $N$. 

Let  $m\geq n$. By orthonormality, 
\begin{align*}
    1=\frac{n!}{m!}\int_0^\infty t^{m-n}\big(L_n^{m-n}(t)\big)^2e^{-t}dt,
\end{align*}
and we deduce that
$$
C_{n,m}(D_R(0))=\frac{n!}{m!}\int_0^{\pi R^2} t^{m-n}\big(L_n^{m-n}(t)\big)^2e^{-t}dt=1-\frac{n!}{m!}\int_{\pi R^2}^\infty t^{m-n}\big(L_n^{m-n}(t)\big)^2e^{-t}dt.
$$
It is therefore sufficient to provide an upper bound for the integral on the right-hand side.

The following basic estimate    was shown in  \cite[Lemma~2.4]{haimi-hedenmalm}  
 \begin{equation}\label{eq:haimi-hedenmalm}
    |L^{m-n}_n(t)|\leq \frac{t^n}{n!},\qquad \text{for }\ t\geq  \frac{n+m}{2}-1+ \sqrt{\tfrac{1}{4}+(n-1)(m-1)} .
    \end{equation}
If $\pi R^2\geq 2N\geq m+n$, for $m,n\leq N$, then   
   \begin{equation}\label{eq:power-of-ball}
    \frac{(\pi R^2)^k}{k!}=\frac{\pi R^2}{k}\, \frac{(\pi R^2)^{k-1}}{(k-1)!}\geq\frac{(\pi R^2)^{k-1}}{(k-1)!},\qquad k\in\{1,...,m+n\}.
    \end{equation}
  Using the recurrence relation $\Gamma(s+1,x)=s\Gamma(s,x)+x^se^{-x}$ for the  upper incomplete gamma function we observe that  
\begin{equation}\label{eq:incomplete-gamma}
    \Gamma(m,t)=(m-1)!\,e^{-t}\sum_{k=0}^{m-1}\frac{t^k}{k!},\qquad m\in\N.
\end{equation}
 Note that for  $t=\pi R^2\geq 2N$, the condition in  \eqref{eq:haimi-hedenmalm} is satisfied, which allows us to combine \eqref{eq:haimi-hedenmalm} \eqref{eq:power-of-ball} and \eqref{eq:incomplete-gamma}  to show  that for $\pi R^2\geq 2N$
    \begin{align*}
 \frac{n!}{m!}\int_{\pi R^2}^\infty t^{m-n}\big(L_n^{m-n}(t)\big)^2e^{-t}dt&\leq \frac{1}{n!m!}\int_{\pi R^2}^\infty t^{m+n}e^{-t}dt    =\frac{\Gamma(n+m+1,\pi R^2)}{n!m!}   
 \\
 &= \frac{(n+m)!}{n!m!}e^{-\pi R^2}\sum_{k=0}^{n+m}\frac{(\pi R^2)^k}{k!}\\
 &\leq \frac{(n+m)!}{n!m!}e^{-\pi R^2}(m+n+1)\frac{(\pi R^2)^{m+n}}{(m+n)!}
 \\
 &\leq (n+m+1)e^{-\pi R^2}\frac{(\pi R^2)^{m+n}}{n!m!} \leq (2N+1)e^{-\pi R^2}\frac{(\pi R^2)^{2N}}{(N!)^2}.
    \end{align*}
Altogether we thus have
\begin{align*}
B(D_R(0))&=\min_{m\in\{1,...,N\}} A_{m}(D_R(0))= \min_{m\in\{1,...,N\}} \sum_{n=0}^N|\lambda_n|^2C_{n,m}(D_R(0))\\&\geq  \min_{m\in\{1,...,N\}} \sum_{n=0}^N|\lambda_n|^2\Big(1-(2N+1)e^{-\pi R^2}\frac{(\pi R^2)^{2N}}{(N!)^2}\Big)\\ &=1-(2N+1)e^{-\pi R^2}\frac{(\pi R^2)^{2N}}{(N!)^2}.
\end{align*}
The final expression is not positive for all $\pi R^2\geq 2N$. However, if we set $\pi R^2=\alpha N$, then we obtain from an application of the pointwise Stirling bound $n!\geq  \sqrt{2\pi n}\left(\frac{n}{e}\right)^n$ (see, e.g., \cite{rob}) that 
    \begin{align}
    B(D_R(0))&\geq 1-(2N+1)e^{-\alpha N}\frac{(\alpha N)^{2N}}{(N!)^2}
 \notag\\ &  \ge 1-\frac{(2N+1)}{2\pi N}\alpha^{2N}e^{N(2-\alpha)} \ge 1-\frac{\alpha^{2N}}{2 } e^{N(2-\alpha)},\label{eq:B-estimate-polyradial}
    \end{align}
    which is greater than zero for $\alpha\geq 5$. 
\end{exmp}
\begin{remark}
    The above estimate in fact works for \(d>1\) as well. The crucial assumption is that the Reinhardt domain \(\Delta\) must be so large that it contains the entire oscillating region given in \eqref{eq:haimi-hedenmalm}. To this end, assuming that \(\Delta\) contains the polydisk of polyradius \((\alpha N,\alpha N,\dots,\alpha N)\) where \(\alpha \geq 5\) and  \( 2N> m_i+n_i \) for all \(i\in \{1,2,\dots d\}\) yields the estimate \begin{align*}
        \frac{n!}{m!}\int_{W^c} t^{n-m} \left(L_n^{m-n}(t)\right)^2e^{-t}dt\leq\left(\frac{2N+1}{2\pi N}\right)^de^{(2-\alpha)d N}\alpha^{2Nd},
    \end{align*}
    where \(W\) is the Reinhardt shadow of \(\Delta\).
\end{remark}

\begin{exmp}
Let again $d=1$, and consider the normalized projection operator $\w=\frac{1}{\sqrt{N}}\sum_{n=0}^{N-1}h_n\otimes h_n$. Then 
$$
A_m(D_R(0))=\frac{1}{N}\sum_{n=0}^{N-1}\frac{n!}{m!}\int_0^{\pi R^2} t^{m-n}\big(L_n^{m-n}(t)\big)^2e^{-t}dt.
$$
As $\{V_{h_m}h_n\}_{n\in\N_0}$ is an orthogonal basis for the reproducing kernel Hilbert space $V_{h_m}(L^2(\R))$ it follows that 
 $$
\sum_{n\in\N_0}V_{h_m}h_n(z)\overline{V_{h_m}h_n(w)}=\langle \pi(w)h_m,\pi(z)h_m\rangle. 
 $$
 In particular,  $$
\int_{D_R(0)}\sum_{n\in\N_0}\frac{n!}{m!}\int_0^{\pi R^2} t^{m-n}\big(L_n^{m-n}(t)\big)^2e^{-t}dt= \int_{D_R(0)}\sum_{n\in\N_0}|V_{h_m}h_n(z)|^2dz=\pi R^2.
 $$
 If we fix  
 $R$ and allow for the flexibility of choosing  $N$ then  there exists $N=N(R)$ such that 
 $$
  \int_{D_R(0)}\sum_{n=N}^\infty|V_{h_m}h_n(z)|^2dz\leq \frac{\pi R^2}{2},
 $$
 which implies that 
 $$
 B(D_R(0))\geq \frac{\pi R^2}{2N}.
 $$
\end{exmp}

\begin{exmp}
Estimating $\theta_\mathcal{K}$ via a lower bound on $B(\Delta)$ often  leads to  rather bad estimates. Take, for example, a small perturbation of the rank-one operator $\w_0=h_0\otimes h_0$, i.e., consider $\w=\sum_{n\in\N^d_0}\lambda_n( h_n\otimes h_n)$ with $\lambda_0=\sqrt{1-\varepsilon}$ and $\lambda_{k}=\sqrt{\varepsilon/N}$ for $k\in \mathcal{N}\subset \N_0^d$, $\#\mathcal{N}=N$, and $\lambda_k=0$ otherwise. For small $\varepsilon$ one has that
$$
\|\V_{\widetilde{\w}(\Delta)} \r\|_{L^1(\R^{2d},\S^2)} \approx C_{0,0}(\Delta)\|\V_{{\w}} \r\|_{L^1(\R^{2d},\S^2)},
$$ 
while   
$$
B(\Delta)=\inf_{m\in\mathcal{N}}\Big((1-\varepsilon)C_{0,m}(\Delta)+\frac{\varepsilon}{N}\sum_{n\in\mathcal{N}}C_{n,m}(\Delta)\Big)\leq  \inf_{m\in\mathcal{N}}C_{0,m}(\Delta)+\varepsilon,
$$
which can be significantly smaller than $C_{0,0}(\Delta)$, in particular, if $\mathcal{N}$ contains an element $k$ with $|k|$ large.  
\end{exmp}

For $\Delta,\Omega\subset\R^d$, the   maximum Nyquist density of $\Omega$, first introduced in \cite{MaxNyquist}, is given by 
$$
\nu(\Omega,\Delta)=\sup_{z\in \R^{2d}}|\Omega\cap (z+\Delta)|.
$$
When \(\Delta=D_R(0)\) is a disk, then \(\nu(\Omega,\Delta)=\nu(\Omega,R)\) is the $R$-measure introduced in the introduction, and \(\nu(\Omega,\Delta)\) can in general   be thought of as a measure of the sparsity of \(\Omega\), see the discussion in Section~\ref{subsec:intro1}. 
With the preparatory computations in place, we are now ready to apply Proposition~\ref{QuantumSieve} in order to connect the best concentration problem to the maximum Nyquist density.
\begin{corollary}\label{cor:max-nyq-LS}
Let $\w=\sum_{n\in \N_0^d} \lambda_n (h_n\otimes h_n)\in \A^1$  be such that $\|\w\|_{\S^2}=1$ and let \(\Omega\subset\R^{2d}\) be measurable. Moreover, let $\theta_\mathcal{K}$ be given by \eqref{eq:theta_k-polyradial}. Then for \(\r\in \A^p\),  it holds 
\begin{align*}
        \frac{\|\chi_{\Omega}\cdot\V_\w\r\|_{L^p(\R^{2d},\S^2)}^p}{\|\V_\w\r\|_{L^p(\R^{2d},\S^2)}^p}&\leq \theta_{\mathcal{K}}\cdot \sup_{z\in \R^{2d}}\int_{\Omega\cap z+\Delta}\left\|\w^*\pi(z)^*\pi(w)\w\right\|_{\mathrm{op}}dw
        \\
        &\leq \theta_{\mathcal{K}}\cdot \sup_{z\in \R^{2d}}\int_{\Omega\cap z+\Delta} \left(\sum_{m,n\in \N_0^d}|\lambda_n|^2|\lambda_m|^2|\langle  h_n,\pi(z-w)h_m\rangle|^2\right)^{1/2}dw
       \\
       &\leq \theta_{\mathcal{K}}\cdot\nu(\Omega,\Delta).
\end{align*} 
\end{corollary}
\begin{proof}
The first inequality  is a direct consequence of  Proposition~\ref{QuantumSieve} applied for $\mathcal{K}(z,w)=K_\w(z,w)\cdot\chi_{z+\Delta}(w)$.
To obtain the second inequality, we recall that  that $\|\cdot\|_{\mathrm{op}}\leq \|\cdot\|_{\S^2}$ and observe that by \eqref{eq:kernel-written-out}
\begin{equation*}\label{eq:kernel-norm}
\|K_\w(z,w)\|_{\S_2}^2=\sum_{m,n\in \N_0^d}|\lambda_n|^2|\lambda_m|^2|\langle \pi(w)h_n,\pi(z)h_m\rangle|^2. 
\end{equation*} 
The third inequality is due to the bound \begin{align*}
    \sum_{m,n\in \N_0^d}|\lambda_n|^2|\lambda_m|^2|\langle \pi(w)h_n,\pi(z)h_m\rangle|^2\leq \sum_{m,n\in \N_0^d}|\lambda_n|^2|\lambda_m|^2=1.
\end{align*}
\end{proof}
Combining Corollary~\ref{cor:max-nyq-LS} with  the estimate from \eqref{eq:B-estimate-polyradial} then establishes  Theorem~\ref{thm:1}.
Moreover, when \(\w=h_n\otimes h_n\) Corollary~\ref{cor:max-nyq-LS} recovers \cite[Theorem~3]{LargeSieve}, as expected.

\section{Applications of the quantum large sieve principle}\label{sec:applications}
In this section we will consider two special cases of the operator STFT, both of which are popular phase space representations of operators and functions, respectively. We will see how the large sieve principle applies to both cases. 

\subsection{The Husimi function}


Given a positive trace class operator \(\r\), the assignment \begin{align*}
    H_\r(z):= \langle \r\,\pi(z)h_0,\pi(z)h_0 \rangle 
\end{align*}
is called the \textit{Husimi Q-function} of \(\r\) \cite{Husimi1, Husimi2, Schupp}, the \textit{coherent state transform} of \(\r\) \cite{FrankNicolaTilli} or the \emph{Berezin transform} of $\r$ \cite{bayer,CorderoGrochenig}. In physics, it is a common way to represent a quantum state on phase space.
Naturally, we interpret the integral \begin{align*}
    \int_{\Omega} H_\r(z)dz
\end{align*}
as the time-frequency localization of the state \(\r\) on the domain \(\Omega\).
An application of Parseval's theorem shows that the Husimi function of \(\r\) equals the convolution \(\r\star (h_0\otimes h_0)\). Keeping in mind that \((h_0\otimes h_0)^2=h_0\otimes h_0\), it is also clear that 
\begin{align*}
\big\|\V_{h_0\otimes h_0}\sqrt{\r}(z)\big\|_{\S^2}^2&=\big\langle (h_0\otimes h_0)\pi(z)^\ast \sqrt{\r },(h_0\otimes h_0)\pi(z)^\ast \sqrt{\r}\big\rangle_{\S^2}
\\
&=\text{tr}\big(\sqrt{\r}\pi(z) (h_0\otimes h_0)\pi(z)^\ast \sqrt{\r}\,  \big)=\text{tr}(\r\, \alpha_z (h_0\otimes h_0) )
\\
&=\sum_{n\in \N_0^d}\langle \r \, \pi(z)h_0,e_n\rangle \langle e_n,\pi(z)h_0\rangle =H_\r(z).
\end{align*}
 Proposition~\ref{QuantumSieve} therefore provides an estimate for the Husimi function's concentration on \(\Omega\) \begin{align}\label{eq:husimi-p}
    \frac{\int_{\Omega}H_\r(z)^{\frac{p}{2}}dz} {\|H_\r\|_{p/2}^{p/2}}\leq \theta_\mathcal{K}\cdot \sup_{w\in \R^{2d}}\int_\Omega \left\|\mathcal{K}(z,w)\right\|_{\mathrm{op}}dz.
\end{align}
 Using our results we can get an estimate 
 in terms of the reproducing kernel, or even in terms of the maximum Nyquist density.
For instance, setting \(p=2\) yields
\begin{align*}
       \int_{\Omega}H_\r(z)dz\leq\theta_\mathcal{K}\cdot \sup_{w\in \R^{2d}}\int_\Omega \left\|\mathcal{K}(z,w)\right\|_{\mathrm{op}}dz\cdot\|\r\|_{\S^1} ,
\end{align*}
where we used Proposition~\ref{prop:op-stft} and the fact that $\|\sqrt{\r}\|_{\S^2}^2=\|\r\|_{\S^1}$.
Applying Corollaries~\ref{QuantumSieve_1} and \ref{cor:max-nyq-LS} to the estimate above, we can explicitly write
\begin{align*}
       \int_{\Omega}H_\r(z)dz\leq  \sup_{w\in \R^{2d}}\int_\Omega e^{-\pi \frac{|z-w|^2}{2}}dz\cdot\|\r\|_{\S^1} ,
\end{align*}
and
\begin{align*}
       \int_{\Omega}H_\r(z)dz\leq   C_{0,0}(\Delta)^{-1}\cdot \nu(\Omega,\Delta)\cdot\|\r\|_{\S^1} .
\end{align*}
The Husimi function therefore distributes an operator 
$\r$
 in phase space in precisely the same way that the spectrogram with a Gaussian window distributes a function, see \cite{LargeSieve}. 

For \(p\geq 2\), the estimate \eqref{eq:husimi-p} can be interpreted as a local  bound on the  'generalized Wehrl entropy' for \(\rho\). The standard sharp bounds given in~\cite{Frank, FrankNicolaTilli} are of the following form: for $\Phi:[0,1]\to \R$ convex one has
$$
\int_{\R^{2}} \Phi(H_\rho(z))dz\leq \int_0^1\Phi(s)\frac{ds}{s},\qquad \r\succeq 0,\ \text{tr}(\r)=1,
$$
with the inequality being sharp.
This was first established by Lieb and Solevej \cite{Lieb-Solovej} for rank-one operators and later extended by Frank \cite{Frank} for density operators. For pure states \(f\otimes f\), a sharp local estimate that extended the results from \cite{faber-krahn} was given in \cite[Theorem~5.3]{KNOcT}. 
For $\Phi(s)=s^{p/2}$, \eqref{eq:husimi-p} gives a new way to locally bound the Wehrl entropy which, in case that the maximum Nyquist density of \(\Omega\) is small,   can significantly improve upon the bounds  in \cite{KNOcT} while also allowing for general density operators $\rho$.

\subsection{Concentration estimates for Cohen's class distributions}
In time-frequency analysis, particular attention is devoted to  quadratic time-frequency distributions that are sesquilinear, covariant and separately weak-* continuous, that is to distributions \(Q\) that satisfy\begin{align*}
    Q(\pi(z_0)f,\pi(z_0)g)(z)&=Q(f,g)(z-z_0),\qquad z,z_0\in\R^{2d},
    \end{align*}
    and    
    \begin{align*}
    \lim_{n\to \infty} \langle Q(f_n,g),\Phi\rangle=\langle Q(f,g),\Phi\rangle,\qquad \lim_{n\to \infty} \langle Q(g,f_n),\Phi\rangle=\langle Q(g,f),\Phi\rangle,
\end{align*}
for all \(\Phi\in \mathscr{S}'(\R^{2d})\) and \(f_n,f,g\in \mathscr{S}(\R^d)\) with \(f_n\to f\),
 see for instance \cite{BoggiattoCohen,RT1,RT2,JanssenSurvey,Cohen}. If $f=g$ we write $Q(f)=Q(f,f).$
The collection of these time-frequency distributions, which is called \textit{Cohen's class}, is closely connected to quantum harmonic analysis. In fact, for any distribution \(Q\) in Cohen's class there is a tempered operator \(\eta\in \mathfrak{S}'\) such that\begin{align*}
    Q(f,g)=Q_\eta(f,g)=(f\otimes g)\star\widecheck{\eta},
\end{align*}
see~\cite[Prop. 7.1]{Bible2}. 
Cohen's class has an important connection to the mixed-state localization operators. Given a function \(f\in L^2(\R^d)\) and a domain \(\Omega\subset \R^{2d}\) we have~\cite[Prop. 8.2]{Bible2}
\begin{align*}
    \langle\chi_{\Omega}\star \eta(f),f\rangle=\int_{\Omega} Q_\eta(f)(z)\,dz.
\end{align*}
So by the min-max principle, the eigenfunctions of the operator \(\chi_{\Omega}\star \eta\) are the stationary points of the best concentration problem \(\displaystyle \sup_{f\in L^2(\R^d)}\|f\|_2^{-2}{\int_{\Omega} Q_\eta f(z)\,dz}\). In particular, when \(\|\w\|_{\S^2}=1\) and \(\eta=\w\w^*\), then the eigenfunctions of the mixed-state localization operator \(A_{\Omega}^{\w\w^\ast}\) are the stationary points of the best concentration problem for the positive Cohen's class distribution \(Q_{\w\w^*}f\).

Let $g\in L^2(\R^d)$ with \(\|g\|_2=1\), and  \(\eta=\w\w^*\). Since \( (f\otimes g) (f\otimes g)^* = f\otimes f,\)  we get that by Lemma~\ref{Identity}
\begin{align}
   Q_{\w\w^*}(f)(z)&= (f\otimes f)\star\widecheck{\w\w^*}=\big((f\otimes g) (f\otimes g)^*\big)\star\widecheck{\w\w^*}
   \notag
   \\ & =   \|\V_\w(f\otimes g)(z)\|_{\S^2}^2=\|\V_\w f(z)\|_{2}^2.  \label{eq:Q_S=V_S}
\end{align}
We can thus use the large sieve principle to obtain concentration estimates for Cohen's class distributions, at least in the case when \(\w\) is admissible. Note that the large sieve approach works even in the case \(p\neq 2\), unlike the min-max approach. Since the operator STFT of an admissible operator defines an equivalent norm on \(M^p(\R^d)\) we get the following corollary for Cohen's class on modulation spaces.

\begin{corollary}\label{CohenSieve}
    Let \(p\in[1,\infty)\), \(\w\in \A^1\) with \(\|\w\|_{\S^2}=1\) and let \(\Omega\subset\R^{2d}\) be measurable. If \(\mathcal{K}\) is as in Proposition~\ref{prop:abstract-LS}, then for \(f\in M^p(\R^d)\) we have\begin{align*}
        \frac{1} {\| Q_{\w\w^*}(f)\|_{p/2}^{p/2}}\int_{\Omega}\left(Q_{\w\w^*}(f)(z)\right)^{\frac{p}{2}}dz\leq \theta_{\mathcal{K}}\cdot \sup_{w\in \R^{2d}}\int_\Omega \left\|\mathcal{K}(z,w)\right\|_{\mathrm{op}}dz.
    \end{align*}
\end{corollary}
\begin{remark}
    If \(\w=g\otimes g\) with \(\|g\|_{L^2}=1\), then $Q_{\w\w^*}(f)(z)=|V_gf(z)|^2$ and we recover the original large sieve principle for the STFT \cite{LargeSieve}.
\end{remark}

If we want to derive a concentration bound for Cohen's class of admissible operators $\eta$ that are not necessarily positive, we can use
\begin{align*}
   |Q_\eta f(z)|&=\big|\big\langle  \eta\pi(z)^\ast f,\pi(z)^\ast f\big\rangle\big| \leq \big\|  \eta\pi(z)^\ast f\big\|_2\,\big\|\pi(z)^\ast f\big\|_2 =\big\|\mathfrak{V}_{\eta^\ast}f(z)\big\|_2\, \|f\|_2.
\end{align*}
to get a simpler bound which is structurally similar to a local Lieb's inequality~\cite{LiebIneq}, although with worse bounds.
\begin{prop}
    Let $p\in[2,\infty)$, $\eta^\ast\in\A^1$ with $\|\eta\|_{\S^2}=1$ and $\Omega\subset \R^{2d}$ measurable. For $f\in L^2(\R^d)$, it holds
    \begin{align*}
\frac{\int_\Omega |Q_\eta f(z)|^pdz}{\|f\|^{2p}_2}&
\leq  \theta_\mathcal{K} \cdot\sup_{w\in\R^{2d}}\int_\Omega \|\mathcal{K}(z,w)\|_{\mathrm{op}}dz.
\end{align*}
\end{prop}
\begin{proof}
Let us assume that $\|f\|_2=1$. Since $\|\V_{\eta^\ast}f(z)\|_2\leq \|\eta^\ast\|_{\mathrm{op}}\|f\|_2\leq 1$, it  then follows that for $p\geq 2$ 
\begin{align*}
\int_\Omega |Q_\eta f(z)|^pdz&\leq \int_\Omega \|\mathfrak{V}_{\eta^\ast} f(z)\|_{2}^pdz
\\
&\leq \theta_\mathcal{K}\cdot \sup_{w\in\R^{2d}}\int_\Omega \|\mathcal{K}(z,w)\|_{\mathrm{op}}dz \cdot\int_{\R^{2d}}\|\mathfrak{V}_{\eta^\ast} f(z)\|_{2}^pdz
\\
&\leq  \theta_\mathcal{K} \cdot\sup_{w\in\R^{2d}}\int_\Omega \|\mathcal{K}(z,w)\|_{\mathrm{op}}dz \cdot\int_{\R^{2d}}\|\mathfrak{V}_{\eta^\ast} f(z)\|_{2}^2dz
\\ &= \theta_\mathcal{K} \cdot\sup_{w\in\R^{2d}}\int_\Omega \|\mathcal{K}(z,w)\|_{\mathrm{op}}dz.
\end{align*}
Rescaling $f$ then yields the stated result.
\end{proof}


An immediate consequence of Corollary~\ref{CohenSieve} is a variant of an uncertainty principle for Cohen's class distributions given in \cite[Proposition~7.7]{Bible2} which can be reformulated as follows:  If \(\eta\) is a bounded operator with $\|\eta\|_{\mathrm{op}}\leq 1$, \(\Omega\subset\R^{2d}\) is a measurable subset and \(f\in L^2(\R^d)\), $\|f\|_2=1$, is such that \begin{align*}
    \int_{\Omega} |Q_\eta f(z)|dz\geq(1-\varepsilon),
\end{align*}
then\begin{align*}
     1-\varepsilon \leq \|\eta\|_{\mathrm{op}}|\Omega|.
\end{align*}
So heuristically, if a large part of the signal's energy is concentrated on a domain \(\Omega\), then that domain must also be large. While the above inequality is true for any bounded operator \(\eta\), the total energy \(\int_{\R^{2d}}|Q_\eta f(z)|\,dz\) will in general only be finite in the case when \(\eta\in \S^1\). Moreover, \(Q_\eta f\) is only positive if \(\eta\) is positive definite, so the above inequality really only behaves like a proper uncertainty principle in this case. In the sequel we will therefore only consider the "physical" case when \(\eta\in \S^1\), and \( \eta\succeq 0.\) Under these reasonable assumptions we extend the result above to  $p\neq2$  and obtain an uncertainty principle that actually measures the concentration of the total energy on $\Omega.$

\begin{prop}\label{UP}
    Let \(p\in[1,\infty),\ \eta\in \S^1\) be a positive operator with \(\|\eta\|_{\S^1}=1\) and $\sqrt{\eta}\in\A^1$. Let \(f\in M^p(\R^d)\) be such that \(\| {Q_{\eta} f}\|_{p/2}=1\), and let \(\Omega\subset \R^{2d}\) and \(\varepsilon\geq 0\).   If  
    \begin{align*}
        \int_{\Omega} \left(Q_{\eta }f(z)\right)^{\frac{p}{2}}dz\geq 1-\varepsilon, 
    \end{align*}
    then we have 
    \begin{align*}
        1-\varepsilon\leq \sup_{w\in \R^{2d}}\int_\Omega \big\|\sqrt{\eta}\pi(z)^*\pi(w)\sqrt{\eta}\,\big\|_{\mathrm{op}}dz\leq\|\eta\|_{\mathrm{op}}|\Omega|\leq |\Omega|.
    \end{align*}
\end{prop}
\begin{proof}
    Using Corollary~\ref{CohenSieve} with \(\w=\sqrt{\eta}\) and $\mathcal{K}=K_\w$, we get \begin{align*}
        1-\varepsilon &\leq  \int_{\Omega} \left(Q_{\eta}f(z)\right)^{\frac{p}{2}}dz\leq \sup_{w\in \R^{2d}}\int_\Omega \big\|\sqrt{\eta}\pi(z)^*\pi(w)\sqrt{\eta}\,\big\|_{\mathrm{op}}dz.      
    \end{align*}
    The result then follows immediately from $\|\sqrt{\eta}\|_{\mathrm{op}}^2=\|\eta\|_\mathrm{op}\leq\|\eta\|_{\S^1}.$
\end{proof}


While the large sieve principle improves the uncertainty principle in the physical case, it cannot improve \cite[Proposition~7.7]{Bible2} in all cases, due to the admissibility assumption.

\section{Signal recovery from operator STFT measurements}\label{sec:recovery}

 Following the ideas of Donoho and Logan \cite{DonohoLogan}, the inequalities established in the preceding sections can be applied to derive signal recovery results. In our case the signals are operators, so the signal recovery task can be thought of as an infinite dimensional matrix recovery problem. Restricting our attention to positive operators with trace \(1\), the task becomes a problem of quantum state reconstruction. We assume that we observe some noisy version of a signal we want to recover. Under certain structural assumptions on the noise we will derive estimates for signal recovery via minimization in the \(1\)-norm. This is reminiscent of theory of compressed sensing \cite{CompressedBook,DonohoCompressed,candes2006robust}. As in \cite{LargeSieve} we  consider two scenarios, although we believe that several other recovery results for functions can be adapted to our setting. 

First, we   consider a scenario  where the noise  is supported on some domain \(\Omega\) but is allowed to have arbitrary  finite \(L^1(\R^{2d},\S^2)\)-norm. Concretely, we assume that we observe the operator-valued function \begin{align}\label{eq:rec-1}
    G(z)=\V_\w\r(z)+N(z),
\end{align} where \(\w,\r\in \A^1\), \(\|\w\|_{\S^2}=1\), \(\|N\|_{L^1(\R^{2d},\S^2)}<\infty\), and \(\mathrm{supp}(N)\subseteq \Omega\). Under certain assumptions on   \(\Omega\), perfect signal recovery is possible. If $N$ is only approximately concentrated on $\Omega$, i.e., $\|\chi_{\Omega^c}\cdot N\|_{L^1(\R^{2d},\S^2)}\leq \varepsilon$, then it is possible to approximately recover $\r$. The abstract tool  that we will use is the following Donoho-Stark type result for Bochner spaces.
\begin{prop}\label{absRecovery}
    Let 
    \(B_1\subset L^1(\R^{2d},\S^2)\) be a Banach space,  \(\Omega\subset \R^{2d}\) and write 
    $$
    \delta(\Omega)=\sup_{F\in B_1}\frac{\int_{\Omega}\|F(z)\|_{\S^2}dz}{\int_{\R^{2d}}\|F(z)\|_{\S^2}dz}.
    $$
    Consider \(G=F+N\) where \(F\in B_1\) and $N\in L^1(\R^{2d},\S^2)$ with \(\|\chi_{\Omega^c}\cdot N\|_{L^1(\R^{2d},\S^2)}\leq\varepsilon\). If \(\delta(\Omega)<\frac{1}{2}\) and $\beta(G)$ is any solution to the minimization problem
    \begin{align}\label{eq:min}
   \argmin_{H\in B_1}\|H-G\|_{L^1(\R^{2d},\S^2)},
    \end{align}
    then
    $$
    \|F-\beta(G)\|_{L^1(\R^{2d},\S^2)}\leq \frac{2\varepsilon}{1-2\delta(\Omega)}.
    $$
    Moreover, if $\emph{supp}(N)\subseteq \Omega$, then $F$ is perfectly recovered as the  unique solution of \eqref{eq:min}.
\end{prop}
\begin{proof}
    We will give the proof for  the sake of completeness, but note that it is done exactly as in the scalar valued case \cite{DonohoLogan,DonohoStark}.

    First, we assume that \(F=0\). 
     Since \(\delta(\Omega)<\frac{1}{2}\) we get using the reverse triangle inequality
     \begin{align*}
        \|H&-N\|_{L^1(\R^{2d},\S^2)} 
        = \left\|\chi_{\Omega}\cdot(H-N)\right\|_{L^1(\R^{2d},\S^2)}+\left\|\chi_{\Omega^c}\cdot(H-N)\right\|_{L^1(\R^{2d},\S^2)}
        \\
        &\geq \left\|\chi_{\Omega}\cdot N\right\|_{L^1(\R^{2d},\S^2)}-\left\|\chi_{\Omega^c}\cdot N\right\|_{L^1(\R^{2d},\S^2)}-\left\|\chi_{\Omega}\cdot H\right\|_{L^1(\R^{2d},\S^2)}+\left\|\chi_{\Omega^c}\cdot H\right\|_{L^1(\R^{2d},\S^2)}\\
        &\geq  \left\|  N\right\|_{L^1(\R^{2d},\S^2)}-2 \varepsilon+\left\|  H\right\|_{L^1(\R^{2d},\S^2)}(1-2\delta(\Omega)).
    \end{align*}
 If $N$ is supported on $\Omega$, i.e., $\varepsilon=0$, we immediately see that   \(H=0\) has to be the unique solution to the minimization problem.

 For $\varepsilon>0$ we note that for any optimizer $\beta(G)$ one trivially has $\|\beta(G)-N\|_{L^1(\R^{2d},\S^2)}\leq \|N\|_{L^1(\R^{2d},\S^2)}$ and therefore, after combining this inequality with the estimate above, we arrive at 
 $$
 \|\beta(G)\|_{L^1(\R^{2d},\S^2)}\leq \frac{2\varepsilon}{1-2\delta(\Omega)}.
 $$
 To prove the result in general, we note that \begin{align*}
       \min_{H\in B_1} \|G-H\|_{L^1(\R^{2d},\S^2)}&=   \min_{H\in B_1}\|F+N-H\|_{L^1(\R^{2d},\S^2)}=   \min_{H\in B_1}\|N-H\|_{L^1(\R^{2d},\S^2)},
    \end{align*}
    and any minimizer of the problem on the left-hand side can be written as a translation by $F$ of a minimizer on the right-hand side and vice versa.
\end{proof}

\bigskip Since $N$ is an arbitrary function in $L^{1}(\mathbb{R}^{2d},S^{2})$%
\ we can take, in Proposition \ref{absRecovery}, $F(z)=\mathfrak{V}_{\gamma
}\rho (z)$ and $N(z)=-\chi _{\Omega }\mathfrak{V}_{\gamma }\rho (z)$. This results in the operator version of Logan's
phenomenon mentioned in the introduction.
\bigskip\ 

\begin{corollary}
\label{Logan}Let  $\w \in \mathfrak{M}^{1}$ and $\Omega\subset\R^{2d}$ be such that  ${\Phi }_{\Omega
,\gamma }^{(1)}(\rho )<\frac{1}{2}$ for every $\r\in\A^1$. Then:
\begin{equation*}
\rho =\argmin_{\sigma \in \mathfrak{M}^{1}}\left\Vert \V_\w\sigma \right\Vert
_{L^{1}(\R^{2d},\S ^{2})},\qquad\text{subject to }\mathfrak{V}%
_{\gamma }\sigma |_{\Omega ^{c}}=\mathfrak{V}_{\gamma }\rho |_{\Omega ^{c}}%
\text{.}
\end{equation*}
\end{corollary}

In the second case we again observe a  version of our signal that is corrupted by noise, but in addition parts of the observation are lost. Compared to the first recovery scenario we now  assume that the (operator-valued) perturbation is small in the  \(L^1(\R^{2d},\S^2)\) sense. Thus our observation is \begin{align}\label{eq:recov-2}
    G(z)=\begin{cases}
        \V_\w\r(z)+N(z), &z\in\Omega,\cr 
        0 \, ,&z\notin\Omega,
    \end{cases}
\end{align}
with \(\w,\r\) and \(\Omega\) as before.
In general, we cannot hope to recover the signal perfectly, but the error that we make is proportional to the $L^1(\R^{2d},\S^2)$-norm of the noise. We  rely on another generalized Donoho-Stark result.
\begin{prop}\label{prop:min-2}
    Let 
    \(B_1\subset L^1(\R^{2d},\S^2)\) be a Banach space. For \(\Omega\subset \R^{2d}\) we define \(\delta(\Omega)\) as before. Assume \(\delta(\Omega)<1\) and  consider \begin{align*}
        G(z)=\begin{cases}
            (F+N)(z), &z\notin \Omega\cr
            0\, ,&z\in \Omega
        \end{cases},
    \end{align*}
    where \(F\in B_1\). 
    If \(\beta(G)\) is any solution of the minimization problem
    \begin{align}\label{eq:min-loss}
     \argmin_{H\in B_1} \|\chi_{\Omega^c}\cdot(G-H)\|_{L^1(\R^{2d},\S^2)},
    \end{align}
    then \begin{align*}
        \|F-\beta(G)\|_{L^1(\R^{2d},\S^2)}\leq \frac{2}{1-\delta(\Omega)}\|N\|_{L^1(\R^{2d},\S^2)}.
    \end{align*}
\end{prop}
\begin{proof} Let $\beta(G)$ be a solution to \eqref{eq:min-loss},
then
\begin{align*}
    \|F&-\beta(G)\|_{L^1(\R^{2d},\S^2)}=\|\chi_\Omega\cdot(F-\beta(G))\|_{L^1(\R^{2d},\S^2)}+\|\chi_{\Omega^c}\cdot(F-\beta(G))\|_{L^1(\R^{2d},\S^2)}
    \\&\leq \delta(\Omega)\| F-\beta(G)\|_{L^1(\R^{2d},\S^2)}+\|\chi_{\Omega^c}\cdot(F-G)\|_{L^1(\R^{2d},\S^2)} +\|\chi_{\Omega^c}\cdot(G-\beta(G))\|_{L^1(\R^{2d},\S^2)} 
    \\
    &\leq \delta(\Omega)\| F-\beta(G)\|_{L^1(\R^{2d},\S^2)}+2\|\chi_{\Omega^c}\cdot N\|_{L^1(\R^{2d},\S^2)}  
\end{align*}
which concludes the proof.
\end{proof}

Combining this with our large sieve estimates we get conditions for recovery from noisy operator STFTs.
\begin{corollary}\label{Recovery}
    Let $B_1=\V_\w(\M^1)$, \(G=\V_\w\r+N\) where \(\w,\r\in \A^1\) and \(N\in L^1(\R^{2d},\S^2)\). If 
        $$
\alpha(\Omega):=\theta_\mathcal{K}\cdot\sup_{w\in\R^{2d}}\int_{\Omega}\|\mathcal{K}(z,w)\|_{\mathrm{op}}\, dz<\frac{1}{2},
    $$
\(\|\chi_{\Omega^c}\cdot N\|_{L^1(\R^{2d},\S^2)}\leq\varepsilon\), and $\beta(G)$ is a solution to the \(L^1\)-minimization problem \eqref{eq:min},  then
    $$
    {\big\|\V_\w \r-\beta(G)\big\|_{L^1(\R^{2d},\S^2)}\leq \frac{2\varepsilon  }{1-2\alpha(\Omega)}.}
    $$
    If $\alpha(\Omega)<1$, and $\beta(G)$ is any solution of the $L^1$-minimization problem \eqref{eq:min-loss}, then 
    \begin{align*}
        {\|\V_\w \r- \beta(G)\|_{L^1(\R^{2d},\S^2)}\leq \frac{2}{1-\alpha(\Omega)}\|N\|_{L^1(\R^{2d},\S^2)}.}
    \end{align*}
\end{corollary}
\begin{proof}
    This is a restatement of Propositions~\ref{absRecovery} and \ref{prop:min-2} for the case \(B_1=\V_\w(\A^1)\), plus an application of the operator STFT inversion formula. The large sieve principle guarantees that \(\delta(\Omega)\leq \theta_\mathcal{K}\cdot\sup_{w}\int_{\Omega}\|\mathcal{K}(z,w)\|_{\mathrm{op}}\, dz\).
\end{proof}
 
 {
\begin{remark}
    Since \(\V_\w^\ast\) is bounded from \(L^1(\R^{2d},\S^2)\) to \(\M^1\), it follows that Corollary \ref{Recovery} gives sufficient conditions for recovery of \(\rho\), not just \(\V_\w\rho.\)
\end{remark}
}

We note here that in Corollary~\ref{Recovery} we replaced the condition on \(\delta(\Omega)\) with a condition on \(\alpha(\Omega)=\theta_{\mathcal{K}}\cdot \sup_{w}\int_{\Omega}\|\mathcal{K}(z,w)\|_{\mathrm{op}}\, dz\) which gives an easier to verify, but less optimal, stability estimate. We extensively discussed the question of how to further bound $\alpha(\Omega)$ in
Section~\ref{sec:estimating-theta}.

As a last remark, we make a brief mention of the implications of the results to the problem of quantum state tomography \cite{Gross,schreiber2025tomography}, which was one of our inspirations. Our  results provide theoretical guaranties for the recovery of a state represented by a density operator, from incomplete information in the phase space. The low rank condition of \cite{Gross,GrossIEEE} is replaced by the measure-sparsity discussed in the introduction. But it should be stressed that, at this point, potential applications are  stylized, due to the continuity of the phase space and the errors resulting from the infinite-dimensional nature of the spaces involved (nevertheless, some problems in quantum science require infinite dimensions and an appropriate formalism has recently been developed \cite{yamasaki2025entanglement}).  Further research would be required to make a decisive statement regarding the feasibility of our approach in real physical settings.

 \section*{Acknowledgements}
The authors thank Franz Luef for several clarifications regarding the
historical development of the new and old concepts of quantum harmonic
analysis used in this paper,  which have been incorporated in this version. This research was funded in part by the Austrian Science Fund (FWF)  through the projects 10.55776/PAT8205923 (L.D.A.) and 10.55776/PAT1384824 (M.S.).  E.S.  acknowledges the partial funding received from the foundation "Norges tekniske høgskoles fond".
For open
access purposes, the authors have applied a CC BY public copyright license to any author-accepted manuscript
version arising from this submission.

\printbibliography

@article {LargeSieve,
    AUTHOR = {Abreu, L. D. and Speckbacher, M.},
     TITLE = {Donoho-{L}ogan large sieve principles for modulation and
              polyanalytic {F}ock spaces},
   JOURNAL = {Bull. Sci. Math.},
    VOLUME = {171},
      YEAR = {2021},
}

@article{husain2024concentration,
  title={Concentration inequalities for Paley--Wiener spaces},
  author={Husain, Syed and Littmann, Friedrich},
  journal={Pacif. J. Math.},
  volume={329},
  number={2},
  pages={201--215},
  year={2024},
  publisher={Mathematical Sciences Publishers}
}

@article{cordero2024wigner,
  title={Wigner analysis of operators. Part II: Schr{\"o}dinger equations},
  author={Cordero, Elena and Giacchi, Gianluca and Rodino, Luigi},
  journal={Comm. Math. Phys.},
  volume={405},
  number={7},
  pages={156},
  year={2024},
  publisher={Springer}
}

@article{iosevich2025uncertainty,
  title={Uncertainty principles, restriction, Bourgain's $\Lambda$q theorem, and signal recovery},
  author={Iosevich, Alex and Mayeli, Azita},
  journal={Appl. Comp. Harm. Anal.},
  volume={76},
  pages={101734},
  year={2025},
  publisher={Elsevier}
}

@article{yamasaki2025entanglement,
  title={Entanglement cost for infinite-dimensional physical systems},
  author={Yamasaki, Hayata and Kuroiwa, Kohdai and Hayden, Patrick and Lami, Ludovico},
  journal={Comm. Math. Phys.},
  volume={406},
  number={11},
  pages={277},
  year={2025},
  publisher={Springer}
}

@article{schreiber2025tomography,
  title={Tomography of parametrized quantum states},
  author={Schreiber, Franz J and Eisert, Jens and Meyer, Johannes Jakob},
  journal={PRX Quantum},
  volume={6},
  number={2},
  pages={020346},
  year={2025},
  publisher={APS}
}

@article{halvdansson2023quantum,
  title={Quantum harmonic analysis on locally compact groups},
  author={Halvdansson, Simon},
  journal={J. Funct. Anal.},
  volume={285},
  number={8},
  pages={110096},
  year={2023},
  publisher={Elsevier}
}

@article{luef2018convolutions,
  title={Convolutions for localization operators},
  author={Luef, Franz and Skrettingland, Eirik},
  journal={J. Math. Pures Appl.},
  volume={118},
  pages={288--316},
  year={2018},
  publisher={Elsevier}
}

@article{abreu2010sampling,
  title={Sampling and interpolation in Bargmann--Fock spaces of polyanalytic functions},
  author={Abreu, Lu{\'\i}s Daniel},
  journal={Appl. Comp. Harm. Anal.},
  volume={29},
  number={3},
  pages={287--302},
  year={2010},
  publisher={Elsevier}
}

@article{candes2006robust,
  title={Robust uncertainty principles: Exact signal reconstruction from highly incomplete frequency information},
  author={Cand{\`e}s, Emmanuel J and Romberg, Justin and Tao, Terence},
  journal={IEEE Trans.  Inform. Theory.},
  volume={52},
  number={2},
  pages={489--509},
  year={2006},
  publisher={IEEE}
}

@article {OpSTFT,
    AUTHOR = {D\"orfler, M. and Luef, F. and McNulty, H. and
              Skrettingland, E.},
     TITLE = {Time-frequency analysis and coorbit spaces of operators},
   JOURNAL = {J. Math. Anal. Appl.},
    VOLUME = {534},
      YEAR = {2024},
    NUMBER = {2},
}

@article{OpSTFTACHA,
  title={Local structure and effective dimensionality of time series data sets},
  author={D{\"o}rfler, Monika and Luef, Franz and Skrettingland, Eirik},
  journal={Appl. Comp. Harmon. Anal.},
  volume={73},
  pages={101692},
  year={2024},
  publisher={Elsevier}
}

@article {EqNorms,
    AUTHOR = {Skrettingland, Eirik},
     TITLE = {Equivalent norms for modulation spaces from positive {C}ohen's
              class distributions},
   JOURNAL = {J. Fourier Anal. Appl.},
    VOLUME = {28},
      YEAR = {2022},
    NUMBER = {2},
}

@article{bayer,
author={Bayer, D. and Gr{\"o}chenig, K.},
title={Time-frequency localization operators and a {B}erezin transform},
journal={Integr. Equ. Oper. Theory},
volume=82,
number=1,
pages={95 – 117}, 
year=2015,
}

@article{de2002uniform,
  title={Uniform eigenvalue estimates for time-frequency localization operators},
  author={De Mari, Filippo and Feichtinger, Hans Georg and Nowak, K18957431033},
  journal={J. London Math. Soc.},
  volume={65},
  number={3},
  pages={720--732},
  year={2002},
  publisher={Oxford University Press}
}

@article{spec-dev-2,
title={Spectral deviation of concentration operators on reproducing kernel {H}ilbert spaces},
author={Marceca, F. and Romero, J. L. and Speckbacher, M. and Valentini, L.},
year=2025,
journal={In preparation},
}

@article{ReinhardtArticle,
    author =  {Chakrabarti, D. and Edholm, L. D.},
title={Projections onto $L^p$-{B}ergman spaces of {R}einhardt domains}, journal={Adv. Math.},
volume=451,
year=2024,
}

@book{SCVBook,
    author = {Grauert, H. and Fritzsche, K.},
title={Several {C}omplex {V}ariables},
volume=38,
series={Grad. Texts in Math.},
publisher= {Springer-Verlag, New York-Heidelberg},
year= 1976,
}

@article{faber-krahn,
author={Nicola, F. and Tilli, P.},
title={The {Faber–Krahn} inequality for the short-time {F}ourier transform},
journal={Invent. Math.},
volume=230,
pages={1–30},
year=2022,
}

@article{stable-faber-krahn,
    author = {G{\'o}mez, J. and Guerra, A. and  Ramos, J.P.G. and Tilli, P.},
title={Stability of the {Faber-Krahn} inequality for the short-time {F}ourier transform},
journal={Invent. Math.},
volume=236,
pages={779–836},
year=2024,
}

@article{haimi-hedenmalm,
    author = {Haimi, A. and Hedenmalm, H.},
title={The polyanalytic {G}inibre ensembles},
journal={J. Stat. Phys.},
volume=153,
number=1,
pages={10-47},
year=2013,
}

@article{svela,
    author = {Svela, E.},
    title = {Extensions of {D}aubechies' theorem: {R}einhardt domains, {H}agedorn wavepackets and mixed-state localization operators},
    journal = 	{arXiv:2410.18769},
    year = 2024,
}

@article{QTFA,
    author={Franz Luef and Henry McNulty},
    title={Quantum Time-Frequency Analysis and Pseudodifferential Operators}, 
    journal = {To appear in: Indiana Univ. Math. J.},
    year = {2025}
}

@book {Grochenig,
    AUTHOR = {Gr\"ochenig, Karlheinz},
     TITLE = {Foundations of Time-Frequency Analysis},
 PUBLISHER = {Birkh\"auser Boston, Inc., Boston, MA},
      YEAR = {2001},
}

@incollection {deBruijn,
    AUTHOR = {de Bruijn, N. G.},
     TITLE = {Uncertainty principles in {F}ourier analysis},
 BOOKTITLE = {Inequalities ({P}roc. {S}ympos. {W}right-{P}atterson {A}ir
              {F}orce {B}ase, {O}hio, 1965)},
     PAGES = {57--71},
 PUBLISHER = {Academic Press, New York-London},
      YEAR = {1967},
   MRCLASS = {42.15},
  MRNUMBER = {219977},
MRREVIEWER = {R.\ R.\ Goldberg},
}

@article {Yvon,
    AUTHOR = {Yvon, Jacques},
     TITLE = {Sur les rapports entre la th\'eorie des m\'elanges et la
              statistique classique},
   JOURNAL = {C. R. Acad. Sci. Paris},
    VOLUME = {223},
      YEAR = {1946},
}

@article {Bible2,
    AUTHOR = {Luef, Franz and Skrettingland, Eirik},
     TITLE = {Mixed-state localization operators: {C}ohen's class and trace
              class operators},
   JOURNAL = {J. Fourier Anal. Appl.},
    VOLUME = {25},
      YEAR = {2019},
}

@book {InterpolationSpaces,
    AUTHOR = {Bergh, J\"oran and L\"ofstr\"om, J\"orgen},
     TITLE = {Interpolation Spaces. {A}n Introduction},
    SERIES = {Grundlehren der Mathematischen Wissenschaften},
    VOLUME = {No. 223},
 PUBLISHER = {Springer-Verlag, Berlin-New York},
      YEAR = {1976},
}

@article {DonohoStark,
    AUTHOR = {Donoho, David L. and Stark, Philip B.},
     TITLE = {Uncertainty principles and signal recovery},
   JOURNAL = {SIAM J. Appl. Math.},
    VOLUME = {49},
      YEAR = {1989},
    NUMBER = {3},

}

@article {Werner,
    AUTHOR = {Werner, R.},
     TITLE = {Quantum harmonic analysis on phase space},
   JOURNAL = {J. Math. Phys.},
  FJOURNAL = {Journal of Mathematical Physics},
    VOLUME = {25},
      YEAR = {1984},
    NUMBER = {5},
}

@article {DonohoLogan,
    AUTHOR = {Donoho, D. L. and Logan, B. F.},
     TITLE = {Signal recovery and the large sieve},
   JOURNAL = {SIAM J. Appl. Math.},
    VOLUME = {52},
      YEAR = {1992},
    NUMBER = {2},
}

@phdthesis{Logan,
    author = {Logan, B. F.},
    title = {Properties of high-pass signals},
    school = {Department of Electrical Engineering, Columbia University},
    year = 1965,
}

@book {deGossonBook,
    AUTHOR = {de Gosson, Maurice A.},
     TITLE = {Symplectic methods in harmonic analysis and in mathematical
              physics},
    SERIES = {Pseudo-Differential Operators. Theory and Applications},
    VOLUME = {7},
 PUBLISHER = {Birkh\"auser/Springer Basel AG, Basel},
      YEAR = {2011},
}

@article {Bauer,
    AUTHOR = {Bauer, Heinz},
     TITLE = {Minimalstellen von {F}unktionen und {E}xtremalpunkte},
   JOURNAL = {Arch. Math.},
    VOLUME = {9},
      YEAR = {1958},
}

@book {Hall,
    AUTHOR = {Hall, Brian C.},
     TITLE = {Quantum theory for mathematicians},
    SERIES = {Graduate Texts in Mathematics},
    VOLUME = {267},
 PUBLISHER = {Springer, New York},
      YEAR = {2013},

}

@article{Gross,
author = {Gross, David and Liu, Yi-Kai and Flammia, Steven and Becker, Stephen and Eisert, Jens},
year = {2010},
month = {10},
pages = {150401},
title = {Quantum State Tomography via Compressed Sensing},
volume = {105},
journal = {Phys. Rev. Lett.},
}

@INPROCEEDINGS{MaxNyquist,
  author={Abreu, Luís Daniel and Speckbacher, Michael},
  booktitle={IEEE Proc. SampTA 2017)}, 
  title={A planar large sieve and sparsity of time-frequency representations}, 
  year={2017},
  volume={},
  number={},
  pages={283-287},
}

@article{sphere,
    author = {Hrycak, T. and Speckbacher, M.},
    title = {Concentration estimates for band-limited spherical harmonics expansions via the large sieve principle},
    journal = {J. Fourier Anal. Appl.},
    year = 2020,
volume=26,
}

@article{ls-wavelets,
    author = {Abreu, L. D. and Speckbacher, M.},
    title = {Donoho-{L}ogan large sieve principles for the wavelet transform},
    journal = {Appl. Comp. Harmon. Anal.},
    year = 2025,
volume=74,
pages= 101709,
}

@article {BoggiattoCohen,
    AUTHOR = {Boggiatto, Paolo and Carypis, Evanthia and Oliaro, Alessandro},
     TITLE = {Cohen class of time-frequency representations and operators:
              boundedness and uncertainty principles},
   JOURNAL = {J. Math. Anal. Appl.},
  FJOURNAL = {Journal of Mathematical Analysis and Applications},
    VOLUME = {461},
      YEAR = {2018},
    NUMBER = {1},
}

@article {RT1,
    AUTHOR = {Ramanathan, Jayakumar and Topiwala, Pankaj},
     TITLE = {Time-frequency localization via the {W}eyl correspondence},
   JOURNAL = {SIAM J. Math. Anal.},
  FJOURNAL = {SIAM Journal on Mathematical Analysis},
    VOLUME = {24},
      YEAR = {1993},
    NUMBER = {5},
}

@Inbook{RT2,
author="Ramanathan, Jayakumar
and Topiwala, Pankaj",
editor="Byrnes, J. S.
and Byrnes, Jennifer L.
and Hargreaves, Kathryn A.
and Berry, Karl",
title="Time-frequency localization operators of Cohen's class",
bookTitle="Wavelets and Their Applications",
year="1994",
publisher="Springer Netherlands",
pages="313--324",
}

@incollection {JanssenSurvey,
    AUTHOR = {Janssen, A. J. E. M.},
     TITLE = {Positivity and spread of bilinear time-frequency
              distributions},
 BOOKTITLE = {The {W}igner {D}istribution},
     PAGES = {1--58},
 PUBLISHER = {Elsevier Sci. B. V., Amsterdam},
      YEAR = {1997},
}

@article {Cohen,
    AUTHOR = {Cohen, Leon},
     TITLE = {Generalized phase-space distribution functions},
   JOURNAL = {J. Mathematical Phys.},
  FJOURNAL = {Journal of Mathematical Physics},
    VOLUME = {7},
      YEAR = {1966},
}

@article{Husimi1,
  title={Some Formal Properties of the Density Matrix},
  author={K&ocirc;di Husimi},
  journal={Proc. Phys.-Math. Soc. Jpn.},
  volume={22},
  number={4},
  pages={264-314},
  year={1940},
}

@article{Husimi2,
  title = {Quantum and classical Liouville dynamics of the anharmonic oscillator},
  author = {Milburn, G. J.},
  journal = {Phys. Rev. A},
  volume = {33},
  issue = {1},
  pages = {674--685},
  year = {1986},
  publisher = {American Physical Society},

}

@article{GrossIEEE,
    author = {Gross, D.},
    title = {Recovering low-rank matrices from few coefficients in any basis},
journal={IEEE Trans. Inform. Theory},
volume=57,
number=3,
pages={1548-1566},
    year = 2011,
}

@article{FrankNicolaTilli,
      author={Rupert L. Frank and Fabio Nicola and Paolo Tilli},
      title={The generalized {W}ehrl entropy bound in quantitative form}, 
      journal={J. Eur. Math. Soc.},
      year={2025},
}

@article {Frank,
    AUTHOR = {Frank, Rupert L.},
     TITLE = {Sharp inequalities for coherent states and their optimizers},
   JOURNAL = {Adv. Nonlinear Stud.},
  FJOURNAL = {Advanced Nonlinear Studies},
    VOLUME = {23},
      YEAR = {2023},
    NUMBER = {1},
}

@article {KNOcT,
    AUTHOR = {Kulikov, Aleksei and Nicola, Fabio and Ortega-Cerd\`a, Joaquim
              and Tilli, Paolo},
     TITLE = {A monotonicity theorem for subharmonic functions on manifolds},
   JOURNAL = {Adv. Math.},
  FJOURNAL = {Advances in Mathematics},
    VOLUME = {479},
      YEAR = {2025},
}

@incollection {Schupp,
    AUTHOR = {Schupp, Peter},
     TITLE = {Wehrl entropy, coherent states and quantum channels},
 BOOKTITLE = {The {P}hysics and {M}athematics of {E}lliott {L}ieb -- {T}he 90th
              {A}nniversary. {V}ol. {II}},
     PAGES = {329--344},
 PUBLISHER = {EMS Press, Berlin},
      YEAR = {2022},
}

@article {CorderoGrochenig,
    AUTHOR = {Cordero, Elena and Gr\"ochenig, Karlheinz},
     TITLE = {Time-frequency analysis of localization operators},
   JOURNAL = {J. Funct. Anal.},
  FJOURNAL = {Journal of Functional Analysis},
    VOLUME = {205},
      YEAR = {2003},
    NUMBER = {1},
     PAGES = {107--131},
}

@article{Lieb-Solovej,
    author = {Lieb, E. H. and Solovej, J. P.},
title={Proof of an entropy conjecture for {B}loch coherent spin states and its
generalizations},
journal={Acta Math.},
volume=212,
year=2014,
number=2,
pages={379–398},
}

@article{rob,
author={Robbins, H.},
title={A remark on {S}tirling’s formula},
journal={Amer. Math. Monthly},
volume=62,
year=1955,
pages={26–29},
}

@article {SchwartzOp,
    AUTHOR = {Keyl, M. and Kiukas, J. and Werner, R. F.},
     TITLE = {Schwartz operators},
   JOURNAL = {Rev. Math. Phys.},
  FJOURNAL = {Reviews in Mathematical Physics. A Journal for Both Review and
              Original Research Papers in the Field of Mathematical Physics},
    VOLUME = {28},
      YEAR = {2016},
    NUMBER = {3},
     PAGES = {1630001, 60},
}

@book{feichtinger1983modulation,
  title={Modulation spaces on locally compact abelian groups},
  author={Feichtinger, Hans G},
  year={1983},
  publisher={Technical Report, University of Vienna}
}

@article{montgomery,
    author = {Montgomery, H. L.},
title={The analytic principle of the large sieve},
journal={Bull. Am. Math. Soc.},
volume=84,
year=1978,
pages={547–567},
}

@article{CandesRomb,
author={Candes, E. J. and  Romberg, J.},
title={Quantitative robust uncertainty principles and optimally sparse decompositions},
journal={Found. Comp. Math.},
volume=6,
number=2,
year=2006,
pages={227-254},
}

@article{SparsityTF,
    author ={Pfander, G. E. and  Rauhut, H.},
title={Sparsity in time-frequency representations},
journal={J. Four. Anal. Appl.},
volume=16,
pages={233-260},
year=2010,
}

@article{Fei,
    author = {Feichtinger, H. G.},
title={On a new {S}egal algebra}, 
journal ={Monatsh. Math.},
volume=92,
number=4,
pages={269-289},
year=1981,
}

@article{Seip0,
    author = {Seip, K.},
title={Reproducing formulas and double orthogonality in {B}argmann and {B}ergman spaces},
journal={SIAM J. Math. Anal.}, 
volume=22,
pages={856-876},
year=1991,
}

@book {CompressedBook,
    AUTHOR = {Foucart, Simon and Rauhut, Holger},
     TITLE = {A {M}athematical {I}ntroduction to {C}ompressive {S}ensing},
    SERIES = {Applied and Numerical Harmonic Analysis},
 PUBLISHER = {Birkh\"auser/Springer, New York},
      YEAR = {2013},
     PAGES = {xviii+625},
}

@article {DonohoCompressed,
    AUTHOR = {Donoho, David L.},
     TITLE = {Compressed sensing},
   JOURNAL = {IEEE Trans. Inform. Theory},
  FJOURNAL = {Institute of Electrical and Electronics Engineers.
              Transactions on Information Theory},
    VOLUME = {52},
      YEAR = {2006},
    NUMBER = {4},
     PAGES = {1289--1306},
}

@article {Daubechies,
    AUTHOR = {Daubechies, Ingrid},
     TITLE = {Time-frequency localization operators: a geometric phase space
              approach},
   JOURNAL = {IEEE Trans. Inform. Theory},
  FJOURNAL = {Institute of Electrical and Electronics Engineers.
              Transactions on Information Theory},
    VOLUME = {34},
      YEAR = {1988},
    NUMBER = {4},
}

@article {AccSpec,
    AUTHOR = {Abreu, L. D. and Gr\"ochenig, K. and Romero,
              J. L.},
     TITLE = {On accumulated spectrograms},
   JOURNAL = {Trans. Amer. Math. Soc.},
  FJOURNAL = {Transactions of the American Mathematical Society},
    VOLUME = {368},
      YEAR = {2016},
    NUMBER = {5},
}

@article {LiebIneq,
    AUTHOR = {Lieb, Elliott H.},
     TITLE = {Integral bounds for radar ambiguity functions and {W}igner
              distributions},
   JOURNAL = {J. Math. Phys.},
    VOLUME = {31},
      YEAR = {1990},
    NUMBER = {3},
}

@article{DKLMcN,
      title={Approximation properties of operator coorbit spaces and sparsity classes}, 
      author={Monika Dörfler and Lukas Köhldorfer and Franz Luef and Henry McNulty},
      year={2025},
      journal={arXiv: 2509.16440}    
}

@article {AccCC,
    AUTHOR = {Luef, Franz and Skrettingland, Eirik},
     TITLE = {On accumulated {C}ohen's class distributions and mixed-state
              localization operators},
   JOURNAL = {Constr. Approx.},
  FJOURNAL = {Constructive Approximation. An International Journal for
              Approximations and Expansions},
    VOLUME = {52},
      YEAR = {2020},
    NUMBER = {1},
}

@article {Fulsche,
    AUTHOR = {Fulsche, Robert and Galke, Niklas},
     TITLE = {Quantum harmonic analysis on locally compact abelian groups},
   JOURNAL = {J. Fourier Anal. Appl.},
  FJOURNAL = {The Journal of Fourier Analysis and Applications},
    VOLUME = {31},
      YEAR = {2025},
}

@article {QHABall,
    AUTHOR = {Dawson, Matthew and Dewage, Vishwa and Mitkovski, Mishko and
              \'Olafsson, Gestur},
     TITLE = {Quantum harmonic analysis on the unweighted {B}ergman space of
              the unit ball},
   JOURNAL = {Integral Equations Operator Theory},
  FJOURNAL = {Integral Equations and Operator Theory},
    VOLUME = {97},
      YEAR = {2025},
}

@article {AffineQHA,
    AUTHOR = {Berge, Eirik and Berge, Stine Marie and Luef, Franz and
              Skrettingland, Eirik},
     TITLE = {Affine quantum harmonic analysis},
   JOURNAL = {J. Funct. Anal.},
  FJOURNAL = {Journal of Functional Analysis},
    VOLUME = {282},
      YEAR = {2022},
}

@article{LieQHA,
      title={Weyl Quantization of Exponential Lie Groups for Square Integrable Representations}, 
      author={Stine Marie Berge and Simon Halvdansson},
      year={2025},
      journal={arXiv:2502.18034},
}
\end{document}